\documentclass[notitlepage,11pt,reqno]{amsart}
\usepackage{color}
\usepackage{amssymb,nicefrac,cite,bm,upgreek,comment}
\usepackage[mathscr]{eucal}
\usepackage[utf8]{inputenc}  
\usepackage{chngcntr}

\definecolor{dmagenta}{rgb}{.4,.1,.5}
\definecolor{007}{rgb}{.0,.0,.7}
\definecolor{dred}{rgb}{.5,.0,.0}
\definecolor{dgreen}{rgb}{.0,.5,.0}
\definecolor{dblue}{rgb}{.0,.0,.5}

\definecolor{violet}{rgb}{.3,.0,.9}
\definecolor{orange}{cmyk}{0,.5,.1,.0}
\definecolor{dcyan}{cmyk}{.5,.0,.0,.0}
\definecolor{dyellow}{cmyk}{.0,.0,.5,.0} 
\definecolor{cm}{cmyk}{1,.0,.0,.0}
\numberwithin{equation}{section}
\allowdisplaybreaks

\def\bye{\end{document}} \def\by{\end{proof}\bye}
\DeclareMathOperator{\newmid}{\,:\,}
\DeclareMathOperator{\Tr}{Tr}
\def\mid{\newmid}
\def\mat{\R^d\otimes \R^d}
\newtheorem{theorem}{Theorem}[section]
\newtheorem{lemma}{Lemma}[section]
\newtheorem{proposition}{Proposition}[section]

\theoremstyle{definition}
\newtheorem{definition}{Definition}[section]

\newtheorem{condition}{Condition}[section]

\theoremstyle{remark}
\newtheorem{remark}{Remark}[section]

\newcommand{\grad}{\nabla}

\DeclareMathOperator{\Exp}{\mathbb{E}}
\DeclareMathOperator{\Prob}{\mathbb{P}}
\newcommand{\D}{\mathrm{d}}

\newcommand{\M}{\mathbf{M}}
\newcommand{\R}{\mathbb{R}}
\newcommand{\B}{\mathbb{B}}
\newcommand{\Rd}{\mathbb{R}^d}

\newcommand{\Act}{\mathbb{U}}
\newcommand{\Uadm}{\mathfrak{M}}
\newcommand{\I}{\mathbf{I}}
\newcommand{\Ir}{\mathbf{1}}

\newcommand{\calB}{\mathcal{B}}
\newcommand{\calC}{\mathcal{C}}
\newcommand{\calCup}{\mathcal{C}^\uparrow}
\newcommand{\calF}{\mathcal{F}}
\newcommand{\calH}{\mathcal{H}}
\newcommand{\calI}{\mathcal{I}}
\newcommand{\calL}{\mathcal{L}}

\newcommand{\calS}{\mathcal{S}}
\newcommand{\calT}{\mathcal{T}}

\newcommand{\ph}{\varphi}

\newcommand{\al}{\alpha}

\newcommand{\lam}{\lambda}
\newcommand{\gam}{\gamma}
\newcommand{\kap}{\kappa}
\newcommand{\eps}{\varepsilon}

\newcommand{\del}{\delta}

\newcommand{\bvnorm}[1]{[\kern-0.45ex[\kern0.1ex #1 \kern0.1ex]\kern-0.45ex]}

\newcommand{\abs}[1]{\lvert#1\rvert}
\newcommand{\norm}[1]{\lVert#1\rVert}
\newcommand{\babs}[1]{\bigl\lvert#1\bigr\rvert}

\newcommand{\bnorm}[1]{\bigl\lVert#1\bigr\rVert}

\newcommand{\df}{:=}

\DeclareMathOperator{\dist}{dist}
\DeclareMathOperator{\co}{co}

\begin{document}

\title[HJB with controlled reflections]
{On viscosity solution of HJB equations with state constraints and reflection control}

\author{Anup Biswas}
\address{Department of Mathematics,
Indian Institute of Science Education and Research,
Dr. Homi Bhabha Road, Pune 411008, India.}
\email{anup@iiserpune.ac.in}
\author{Hitoshi Ishii}
\address{Faculty of Education and Integrated Arts and Sciences, 
Waseda University, Shinjuku,Tokyo, 169-8050 Japan.}
\email{hitoshi.ishii@waseda.jp}
\author{Subhamay Saha}
\address{ Department of Electrical Engineering,
Technion, Haifa 32000, Israel.}
\email{subhamay585@gmail.com}
\author{Lin Wang}
\address{Yau Mathematical Sciences Center, Tsinghua University, Beijing, 100084, China.}
\email{lwang@math.tsinghua.edu.cn}

\date{\today}

\begin{abstract}
Motivated by a control problem of a certain queueing network we consider a control problem where the dynamics is constrained in the nonnegative orthant $\Rd_+$
of the $d$-dimensional Euclidean
space and controlled by the reflections at the faces/boundaries. We define a discounted value function associated to this problem and show that the value function is a viscosity
solution to a certain HJB equation in $\Rd_+$ with nonlinear Neumann type boundary condition. Under certain conditions, we also characterize this value function as the unique
solution to this HJB equation.
\end{abstract}

\subjclass[2000]{93E20, 60H30, 35J25, 49L25}

\keywords{Skorokhod map with reflection control, queues with help, viscosity solutions, nonlinear Neumann boundary, non-smooth domain, heavy-traffic, stochastic network.}

\maketitle

\tableofcontents
\section{Introduction}
The chief goal of this article is to characterize the value function for discounted control problem
associated to certain constrained dynamics (see \eqref{1}) as a viscosity solution of certain Hamilton-Jacobi-Bellman (HJB) equation
with nonlinear Neumann type boundary condition. The controlled diffusion takes values in $\Rd_+$,
non-negative orthant of the Euclidean space $\Rd$. The control process is matrix valued and represents
the reflection vectors at certain time. This problem is motived by queueing network problem where
there are $d$ classes of customers in the queueing system with $d$ number of servers, for every server there is a priority class of customers, and customers from a non-priority class may access the server provided its priority class
is empty at that instant (see Section~\ref{motive} for more details). One of the interesting features of this model is its control process which acts through reflection. To the authors knowledge
such model has not been studied before. For existence and uniqueness results for such dynamics we follow \cite{ram}. The constrained dynamics in \cite{ram} does not have any control
component and in fact, the author assumes certain regularity properties on the reflection matrix in the time variable. We show in Theorem~\ref{Thm-existence} that such regularity is not needed to establish existence of a unique adapted solution.

The value function is defined to be the infimum of a certain discounted cost function (see \eqref{3}) where the infimum is taken over all \emph{admissible} controls.  Admissible
controls are processes that do not predict the future. Discounted cost functions are quite common in queueing theory (see \cite{atar-budhiraja} and references therein).
The discounted cost considered here has two components (see \eqref{2.5}), the running
cost $\ell\in\calC_{\mathrm{pol}}(\Rd_+)$ and the cost $h_i\in\calC(\mat)$, with $i\in \calI:=\{1, \ldots, d\}$, associated to the reflections . We show that the value function is in
$\calC_{\mathrm{pol}}(\Rd_+)$ and is a viscosity solution to the following HJB equation (see \eqref{4})
\begin{equation}\label{HJB}
\left\{\begin{split}
\calL V -\beta V + \ell &= 0 \quad  \text{in}\; \; \left(\Rd_+\right)^\circ,
\\
\calH_i (DV(x)) &= 0 \quad \text{ for }x\in \partial \R^d_+,\ 
i\in\I(x),
\end{split}
\right.
\end{equation}
where
$$\calH_i(p)\, \df\, \min_{M\in\M}\{p\cdot M e_i + h_i(M)\}\, , \quad i\in\calI,$$
$\M$ is a certain subset of $\mat$, and $\I(x)$ denotes the collection of indices $i\in\calI$ for which $x_i=0$. 
We refer \cite{crandall-ishii-lions} for more
details on viscosity solutions.
We note that the boundary condition here is of nonlinear Neumann type. The major hurdles in obtaining the uniqueness of viscosity
solution are the non-smooth property of the boundary and the nonlinear nature of the boundary condition.  
The authors of \cite{atar-dupuis, dupuis-ishii-90, dupuis-ishii} study viscosity solution in a non-smooth domain but with
linear Neumann type boundary condition. In \cite{bor-budhi} viscosity solution framework is 
used to characterize the value function of an ergodic control problem with dynamics given by controlled constrained diffusions in  a polyhedral domain.
 Nonlinear Neumann boundary condition is studied in
\cite{barles, ishii-sato} (see also the references therein) for domains with regular boundary (at least $\calC^1$). The key idea in proving uniqueness of viscosity solution is to construct a suitable test function.
We could establish the uniqueness under some assumptions on $h_i$
(see Theorem~\ref{Thm-main}(b)). The construction of test function also uses the structure of nonlinear boundary condition.

The organization of this article is as follows. Section~\ref{S-main} introduces the control problem settings and the main results. Then in Section~\ref{proofs(a)} we establish regularity properties of the value function and prove that the value function is a viscosity solution to \eqref{HJB}, while the uniqueness 
of a viscosity solution to \eqref{HJB} is established in Section~\ref{S-uni} 
under the assumption of the existence of a right test function. 
Section~\ref{const-tf} is devoted to establishing the existence of a test function.     
Finally, some auxiliary results are proved in the Appendix.

{\bf Notations:} By $\Rd$ we denote the $d$-dimensional Euclidean space equipped with the Euclidean norm $\abs{\cdot}$. By $\{e_1, \ldots, e_d\}$ we denote the standard orthonormal basis in $\Rd$. We write $\calI$ for the set $\{1,\ldots,d\}$ and set 
$\Rd_+=\{x=(x_1,\ldots,x_n)\in\R^d : x_i\geq 0 \text{ for all } i\in\calI\}$.  We denote $\Ir=(1,\ldots, 1)\in\R^d_+$. 
We denote by $\mat$ the set of all $d\times d$ real matrices and we endow this space with the usual metric.
For $a, b\in\R$ we denote the maximum (minimum) of $a$ and $b$  as $a\vee b$ ($a\wedge b$, respectively). 
We define $a^+=a\vee 0$ and $a^-=- a\wedge 0$. Trace of a matrix $A\in\mat$ is denoted by $\Tr A$.
For $x\in\Rd_+$ we define $\B_r(x)= \{y\in\Rd_+\, :\, \abs{y-x}<r\}$ for $r>0$.
Given a topological space $\mathcal{X}$ and $B\subset\mathcal{X}$,
the interior, complement and boundary of $B$ in $\mathcal{X}$ is denoted by $B^\circ, B^c$ and $\partial B$, respectively.
By $\calB(\mathcal{X})$ we denote the Borel $\sigma$-field of $\mathcal{X}$. Let $\calC([0, \infty) : \mathcal{X})$
be the set of all continuous functions from $[0, \infty)$ to $\mathcal{X}$. By $\calCup_0([0, \infty) : \Rd_+)$ we denote
the set of all continuous paths that are non-decreasing (component-wise) with initial value $0$. We define $\calC^2(\Rd)$
as the set of all real valued twice continuously differentiable functions on $\Rd$. $\calC_{\mathrm{pol}}(\Rd_+)$ denotes the
set of all real valued continuous functions $f$ with at most polynomial growth i.e.,
$$\limsup_{|x|\to 0}\frac{\abs{f(x)}}{\abs{x}^k}= 0\quad \text{for some }\; k\geq 0.$$
Infimum over empty set is regarded as $+\infty$. $C, \kappa_1, \kap_2, \ldots,$ are deterministic positive constants whose values might change from line to line.

\section{Setting and main result}\label{S-main}
We consider the problem in the domain $\Rd_+$. 
Let $(\Omega, \calF, \Prob)$ be a complete probability space on which a filtration $\{\calF_t\}_{t\geq 0}$
satisfying the usual hypothesis is given. Let $(W_t, \calF_t)$
be a $d$-dimensional standard Wiener process (Brownian motion) on the
above probability space, i.e., the Wiener process $W_t$ is adapted to $\calF_t$, $t\geq 0$,
 and for every $t, s \geq 0$, $W_{t+s}-W_t$ is independent of $\calF_t$. Let $b:\Rd_+\to\Rd$ be a mapping and $\Sigma$ be a $d\times d$ matrix satisfying the following condition.
\begin{condition}\label{cond1}
There exists a $K\in(0, \infty)$ such that
$$\babs{b(x)-b(y)}\leq K \babs{x-y} \quad \text{and}\quad \abs{b(x)}\leq K \quad \text{ for all }\ x, \, y\in\Rd_+.$$
Also the matrix $A=\Sigma\Sigma^T$ is non-singular.
\end{condition}
We note that Condition~\ref{cond1} in particular, implies that $\Sigma$ is non-singular. Let $\al\in [0, 1]^d$. We define
$$\M(\al) \df\, \{M=[\mathbb{I}-P]\in \mat\, :\, P_{ii}=0, P_{ij}\geq 0\,\, \text{and}  \sum_{j\neq i}P_{ji}\leq \alpha_i\},$$
where $\mathbb{I}$ denotes the $d\times d$ identity matrix. We endow that set $\M(\alpha)$ with the metric of $\mat$ and therefore $\M(\al)$ forms a compact
metric space. We assume that
\begin{condition}\label{cond2}
$$\max_{i}\, \al_i<1.$$
\end{condition}
In what follows, $\al$ is assumed to be fixed and now onwards we write $\M(\al)$ as $\M$. 
Given $x\in \Rd_+$ we consider the following reflected stochastic differential equation
\begin{equation}\label{1}
\begin{split}
 X(t)  &= x +\int_0^t b(X(s))\, ds + \Sigma\, W(t) + \int_0^t v(s)\, dY(s),\qquad t\geq 0,
\\
\int_0^t X_i(s)\, dY_i(s) &=  0 \quad \text{for}\; i\in\calI,
\quad \: X\geq 0,\: Y\geq 0, \; t\geq 0\,,
\end{split}
\end{equation}
where the process $v$ takes values in the compact metric space $\M$, and $v(s, \omega)$ is jointly measurable in $(s, \omega)$. We also assume that the control process
$v$ is $\{\calF_t\}$ adapted. It is easy to see that
$v$ is \textit{non-anticipative} : For $s<t$, $W(t)-W(s)$ is independent of
$\sigma\{W(r), v(r)\, :\, r\leq s\}$ as $W(t)-W(s)$ is independent of $\calF_s$. We call the control  $v$ an \textit{admissible} control and the set of all admissible controls is denoted
by $\Uadm$. We show in Theorem~\ref{Thm-existence} in Appendix that if
Conditions~\ref{cond1} and \ref{cond2} hold then for every $v\in\Uadm$, \eqref{1} has a unique adapted strong solution $(X, Y)$ taking values
in $\calC([0, \infty) :\Rd_+)\times\calCup_0([0, \infty) : \Rd_+)$.

It is well known that for every non-decreasing continuous function $\phi :\R_+\to\R_+$, we can define a $\sigma$-finite measure $d\phi$ on $([0, \infty), \calB([0, \infty)))$ where $d\phi((a, b])=\phi(b)-\phi(a)$ for
$0\leq a\leq b<\infty$.  For any $f\in L^1([0, \infty), \calB([0, \infty)), d\phi)$ we denote $\int_{[0, t]}f(s) d\phi(s)= \int_0^t f(s)d\phi(s)$.
Let $\ell$ be a non-negative continuous function with polynomial growth, i.e., there exists $m, c_\ell\in[0, \infty)$ such that
\begin{equation}\label{2}
0\; \leq\; \ell(x)\; \leq\; c_\ell\; (1+\abs{x}^m), \qquad x\in\Rd_+.
\end{equation}
Let $h_i:\M\to\R, i\in\calI$, be continuous. Since $\M$ is compact, $h_i, i\in\calI$, are bounded. Note that $h_i$ can as well
be negative. Then for any initial data $x\in\Rd_+$ and control $v\in\Uadm$, our cost function is given by
\begin{equation}\label{2.5}
 J(x, v)\;\df \;\Exp_x\Big[\int_0^\infty e^{-\beta s}\bigl(\ell(X(s))\, ds +
 \sum_{i\in\calI} h_i(v(s))\, dY_i(s)\bigr)\Big],
 \end{equation}
where $\beta>0$, and $(X, Y)$ is the unique adapted solution to \eqref{1}.
We define the value function as
\begin{equation}\label{3}
V(x) \, \df\, \inf_{v\in\Uadm} J(x, v).
\end{equation}
It is shown in Lemma~\ref{lem-growth} below that $V(x)$ is finite for all 
$x\in\Rd_+$.
The goal of this article is to characterize the value function $V$ as a viscosity solution of a suitable Hamilton-Jacobi-Bellman (HJB) equation.
To define the HJB equation let us first introduce the following notations,
\begin{align*}
\calL  \df \frac{1}{2}\Tr (A D^2) + b \cdot D\, ,
\end{align*}
where $A=\Sigma\Sigma^T$ and $D$ and $D^2$ are the gradient 
and Hessian operators, respectively.  
Define for $p\in\Rd$,
$$\calH_i(p)\, \df\, \min_{M\in\M}\{p\cdot M e_i + h_i(M)\}\, , \quad i\in\calI.$$
For $x\in\Rd_+$ we define $\I(x)\df\{i\in\calI \; :\; x_i=0\}$. 
Therefore $\I(x)$ denotes the indices of faces $F_i:=\{x\in\Rd_+\newmid 
x_i=0\}$ where $x$ lies.
Note that $\I(x)$ can also be empty. 
Indeed, for $x\in\R^d_+$, 
\[
\I(x)\begin{cases}
=\emptyset &
\text{ if } \ x\in\left(\R^d_+\right)^\circ,\\[3pt]
\not=\emptyset&\text{ if }\ x\in\partial \R^d_+. 
\end{cases}
\]
The HJB equation we are interested in is given as follows
\begin{equation}\label{4}
\left\{\begin{split}
\calL V -\beta V + \ell &= 0\quad  \text{in}\; \; \left(\Rd_+\right)^\circ,
\\
\calH_i (DV(x)) &= 0 \quad\text{ for }  x\in 
\partial\R^d_+,\ i\in I(x),
\end{split}
\right.
\end{equation}
The  solution of \eqref{4} is defined in the viscosity sense 
(\hspace{1sp}\cite{crandall-ishii-lions,dupuis-ishii-90,dupuis-ishii}) as follows:
\begin{definition}[Viscosity solution]\label{defi-vis}
Let $V\, :\, \Rd_+\to\R$ be a continuous function. Then $V$ is said to be a sub (super) solution of \eqref{4},
whenever $\varphi\in \calC^2(\Rd)$ and $V-\varphi$ has a local maximum (minimum) at $x\in\Rd_+$, the following
hold:
\begin{align*}
(\calL \varphi(x)-\beta V(x) + \ell(x))\vee \max_{i\in\I(x)}\calH_i(D\varphi(x)) &\geq 0\, ,
\\
\Big((\calL \varphi(x)-\beta V(x) + \ell(x))\wedge \min_{i\in\I(x)}\calH_i(D\varphi(x)) &\leq 0, \ \ \text{respectively}\Big).
\end{align*}
\end{definition}
We will impose the following condition on $h_i$.
\begin{condition}\label{cond3}
 For any $p\in\Rd$ and $x\in\partial\Rd_+$ if we have $\max_{i\in\I(x)}\calH_i(p)<\eta$ for some
$\eta\in(-\infty, 0)$ then
$$\cap_{i\in\I(x)}\{M\in\M\; : \; p\cdot M e_i+ h_i(M)<\eta/2\}\neq\emptyset.$$
\end{condition}
\begin{remark}
Condition~\ref{cond3} is assumed purely for technical reasons. This condition will be used to show that $V$ in
\eqref{3} is a subsolution to \eqref{4}. It is easy to see that Condition~\ref{cond3} is satisfied if
$h_i(M)\equiv h_i(M_{1i}, \ldots, M_{di})$.
\end{remark}

Our main result characterizes the value function $V$  given by \eqref{3} as the viscosity solution of \eqref{4}.
Recall that $\calC_{\mathrm{pol}}(\Rd_+)$ denotes the set of all continuous functions, defined on $\Rd_+$, with polynomial growth.
Our main result is the following.

\begin{theorem}\label{Thm-main}
Let Conditions \ref{cond1}, \ref{cond2} and \ref{cond3} hold. Then the following holds.
\begin{itemize}
\item[{(a)}] The value function $V$ given by \eqref{3} is a viscosity solution
to the HJB given  by \eqref{4};
\item[{(b)}] Assume that  $h_i(M)= c\cdot Me_i$ for some constant vector $c=(c_1,\ldots,c_d)$. Then $V$ is the unique
viscosity solution to \eqref{4} in the class of $\calC_{\mathrm{pol}}(\R^d_+)$.
\end{itemize}
\end{theorem}

Proof of Theorem~\ref{Thm-main}(a) is done in Lemma~\ref{lem-growth}, Lemma~\ref{lem-cont}
and Theorem~\ref{Thm-viscosity}. These results are proved in Section~\ref{proofs(a)}. Let us mention again 
that the uniqueness of viscosity solution does not follow from existing literature.
We construct an appropriate test function, as stated in Theorem \ref{test-f}, in Section~\ref{const-tf}, 
which leads us to the uniqueness result.

\subsection{Motivation from queueing network}\label{motive}
The motivation of the above mentioned problem comes from a control problem in multi-class queueing network.
Consider a queueing system with $d$ customer classes where the arrivals of the customers are given by $d$ independent
Poisson processes $\{E^1, \ldots, E^d\}$. There are $d$ servers and  every customer class has service priority in one of the servers. Therefore we can label the servers from $1$ to $d$ so that customer of class $i$ has priority for service in server $i$ for $i\in\calI$. A customer of class $j, \, j\neq i,$ may access the server $i$ if there are
no customers of class $i$ present in the system. We assume that the network is working under
a preemptive scheduling, i.e., service of a non-priority class customer is preempted by a
priority class customer.
Such queueing systems might be referred to as \emph{queueing networks with help}.
Similar type queueing network is studied in \cite{biswas}
for a M/M/N+M queueing network with ergodic type costs.
The service time distributions of the customers
at different servers are assumed to be exponential. If we denote by $X^i_t$, the number of class $i$ customers in the system at time $t$, then we have the following balance
equation
\begin{equation}\label{E01}
X^i_t\; =\; X^i_0 + E^i_t - S^i(\mu^i \int_0^t B^i_s\, d{s}) - \sum_{j:j\neq i} S^{ij}(\mu^{ij}\int_0^t B^{ij}_s\, d{I^j_s}), \quad i\in \calI. 
\end{equation}
In \eqref{E01}, $\{S^i, S^{ij}, i, j\in\calI\}$ are independent standard Poisson process, $\mu^{ij}$ ($\mu^i$) denotes the service rate of the class $i$ (class $i$)
customers at the server  $j, j\neq i$ ($i$, respectively). $B^{ij}_s$ ($B^i_s$) is an adapted process which denotes the status of service of class $i$ customers at server
$j$ ( $i$, respectively) at time $s$. $B^{ij} =0$ implies that class $i$ customer is not receiving service from server $j$ whereas $B^{ij}\in(0, 1]$ implies that the
class $i$ customer is being served at server $j$ with fraction of effort $B^{ij}$. Therefore we have
$$0\; \leq \; \sum_{i: i\neq j} B^{ij}_s \;\leq\; 1.$$
Due to priority it is easy to see that $B^i\in\{0, 1\}$. $I^i$ denote the cumulative time when there were no customers of class $i$ present in the system, i.e.,
$$I^i_t\; \df\; \int_0^t (1-B^i_s)\, d{s}, \quad i\in \calI.$$ 
Therefore $I^i$ is an increasing process and integrations in \eqref{E01} make sense. We assume that every class is in heavy traffic in the sense that the
arrival rate of class $i$ customers is equal to $\mu^i$. Then we would have $\frac{I^i_{nt}}{n}\to 0$ as $n\to \infty$, and $t\geq 0$.
We define the diffusion scaling as follows.
\begin{equation*}
\begin{gathered}
\Hat{E}^i_t\; \df\; \frac{1}{\sqrt{n}}(E^i_{nt}-\mu^i \, n t), \quad \Hat{S}^i(t) \; \df\; \frac{1}{\sqrt{n}} (S^i(nt) - n t),\quad \Hat{S}^{ij}(t) \; \df\; \frac{1}{\sqrt{n}} (S^{ij}(nt) - n t),
\\[2mm]
\Hat{X}^i_t \; \df\; \frac{1}{\sqrt{n}} X^i_{nt}, \quad \Hat{I}^i_t \; \df\; \frac{1}{\sqrt{n}} I^i_{nt}, \quad \bar{I}^i_t \; \df\; \frac{1}{n} I^i_{nt}.
\end{gathered}
\end{equation*}
Hence using \eqref{E01} and a simple calculation we obtain
\begin{align}\label{E02}
\Hat{X}^i_t &=\; \Hat{X}^i_0 + \Hat{E}^i_t - \Hat{S}^i(\mu^i \int_o^t B^i_{ns}\, d{s}) - \sum_{j: j\neq i} \Hat{S}^{ij}(\mu^{ij}\int_0^t B^{ij}_{ns}\, d{\bar{I}^j_s})\nonumber
\\[2mm]
&\; \qquad + \mu^i \Hat{I}^i_t - \sum_{j: j\neq i} \mu^{ij}\, \int_{0}^t B^{ij}_{ns} \, d{\Hat{I}^j_s},
\end{align}
for $i\in\calI$. We also have
$$\int_0^t \Hat{X}^i_s \, d{\Hat{I}^i_s}\; =\; 0 \quad \text{for all}\; 
i\in\calI, \; t\geq 0.$$
Now formally letting $n\to \infty$, in \eqref{E02} we have: $(\Hat{X}, \Hat{I})\to (Z, \Hat{Z})$ and
\begin{equation}\label{E03}
\begin{split}
Z^i_t &= Z^i_0 + W^{i, 1}_t -W^{i, 2} + \mu^i \Hat{Z}^i - \sum_{j: j\neq i} \mu^{ij}\, \int_0^t v_{ij}(s) \, d{\Hat{Z}^j_s},
\\
\int_0^t Z^i_s \, d{\Hat{Z}^i_s} & =0,\quad \forall \; t\geq 0,
\end{split}
\end{equation}
where $\{W^{i,1}, W^{i, 2}, i\geq 1\}$ are collection of independent Brownian motion with suitable variance. The relation between \eqref{E02} and \eqref{E03} is quite formal, in fact, the convergence might not hold for any choice of control $B^{ij}_n$. However, this relation can be justified for a reasonable class of controls.
Now if we redefine $\Hat{Z}^i$ as $\mu^i \Hat{Z}^i$ then
from \eqref{E03} we get
\begin{equation}\label{E04}
\begin{split}
Z^i_t &= Z^i_0 + W^i + \Hat{Z}^i - \sum_{j: j\neq i}  \int_0^t \frac{\mu^{ij}}{\mu^j}\, v_{ij}(s) \, d{\Hat{Z}^j_s},
\\
\int_0^\cdot Z^i_s \, d{\Hat{Z}^i_s} & =0,\quad \forall \; t\geq 0,
\end{split}
\end{equation}
for some Brownian motion $(W^1, \ldots, W^d)$ with suitable covariance matrix and $v$ takes values in $\M(\Ir)$.
 Assume that $\frac{\mu^{ij}}{\mu^j}< 1$ for all $i, j$. Then it is easy to see that
\eqref{E04} is a particular case of \eqref{1} and Condition~\ref{cond2} is also satisfied.

We can associate the following cost structure to the above queueing system.
$$V(\Hat{X}_0)\;\df\; \inf\, \Exp
\Bigl[\int_0^\infty e^{-\beta s} \, \bigl(\ell(\Hat{X}_s)\, ds + \sum_{i\in\calI} c_i\, d\Hat{I}^i_s\bigr)\Bigr],$$
for some non-negative constants $c_i, i\in\calI$, where the infimum is taken over all adapted process $B$. We note that if we let $n\to\infty$ formally then the above
cost structure is a particular form of \eqref{3}. Cost structure of above type has been used to find optimal control for various queueing networks (see for example 
\cite{atar-budhiraja} and references therein). The importance of such control problems comes from well studied Generalized Processor Sharing (GPS) networks 
\cite{ramanan-reiman-03,ramanan-reiman-08}.
In a single server network working under GPS scheduling in heavy-traffic, the server spends a positive fraction of effort (corresponding to their traffic intensities) to serve each customer class when all classes have backlog of work, otherwise the service capacity is split among the existing classes in a certain proportion. It is easy to see that our model considers a full-fledged optimal control version of the GPS models. 

\begin{remark}
The above relation between the queuing control problem and the value function defined in \eqref{3}
is bit formal. We neither study the convergence of the value functions corresponding
to the queuing model, nor try to find any asymptotic optimality result. These are topics of further
study. The main goal of this article is to characterize the value function \eqref{3} in terms of 
the solution of HJB \eqref{4}. However, in the many server settings above questions are addressed in   \cite{biswas}. The limiting HJB in \cite{biswas} has a classical solution which 
is used to obtain an optimal control for the limiting problem. This optimal control is then  used to obtain convergence result of the value functions for the queuing model and asymptotic optimality. It should be noted that the solution of \eqref{4} is obtained in viscosity sense which gives additional difficulty in establishing convergence results of value functions.
\end{remark}

\section{Proof of Theorem ~\ref{Thm-main}(a)}\label{proofs(a)}
Conditions~\ref{cond1} and \ref{cond2} will be in full effect in the remaining part of the article.
We start by proving the growth estimate of $V$ given by \eqref{3}.
\begin{lemma}\label{lem-growth}
Let $V$ be given by \eqref{3}. Then there exist $c_1\in (0, \infty)$ such that
$$|V(x)|\; \leq\; c_1(1+\abs{x}^m), \quad x\in\Rd_+.$$
\end{lemma}
 
\begin{proof}
Fix $x\in\Rd_+$ and $v\in\Uadm$. Let $(X, Y)$ be the solution of \eqref{1} given the control $v$ and initial data $x$. Then we have
\begin{equation}\label{6}
X_i(t) = x_i + \int_0^t b_i(X(s))\,ds +\sum_{j=1}^d \Sigma_{ij} W_j(t)
+ \int_0^t\sum_{j=1}^d v_{ij}(s)\, dY_j(s), \quad t\geq 0.
\end{equation}
Therefore summing over $i$ we have from \eqref{6} that
\begin{equation}\label{7}
\sum_{i\in\calI} X_i(t) = \sum_{i\in\calI} x_i + \int_0^t \sum_{i\in\calI} b_i(X(s))\, ds +
\sum_{j=1}^d \Big(\sum_{i\in\calI}\Sigma_{ij}\Big) W_j(t) + \sum_{j=1}^d\Big[\int_0^t\sum_{i\in\calI} v_{ij}(s)\, dY_j(s) \Big]
\end{equation}
Since $v$ takes values in $\M$ we have $\sum_{i}v_{ij}\leq 1$, thus using Condition~\ref{cond1}
and \eqref{7} we get a constant $\kap_2>0$ such that
\begin{align*}
\sum_{i\in\calI}X_i(t)\leq \sum_{i\in\calI} x_i + \kap_2\Big[t+\sum_{i\in\calI} \bigl(|W_i(t)| + Y_i(t)\bigr)\Big].
\end{align*}
Now use \eqref{a-3} to obtain  
\begin{equation}\label{8}
\sum_{i\in\calI}X_i(t)\leq \sum_{i\in\calI} x_i + 
\kap_3\Bigl(t + \sum_{i\in\calI} \sup_{0\leq s\leq t}|W_i(s)|\Bigr),
\end{equation}
for some constant $\kap_3$. Recall that for any $a_i\geq 0, i=1,\ldots, d+2,$ and $m\in[0, \infty)$ we have
$$\Bigl(\sum_{i=1}^{d+2} a_i\Bigr)^m\leq (d+2)^{(m-1)^+} \sum_{i=1}^{d+2}a_i^m.$$
Since $X\geq 0$ we have from \eqref{2} and \eqref{8} that
\begin{align}\label{9}
\ell(X(t)) &\leq c_\ell + c_\ell\Bigl[\sqrt{d}\abs{x}+\kap_3 \bigl(t + \sum_{i\in\calI} \sup_{0\leq s\leq t}|W_i(s)|\bigr)\Bigr]^m\nonumber
\\
&\leq c_\ell + c_\ell (d+2)^{(m-1)^+} d^{m/2}\abs{x}^m
+ c_\ell (d+2)^{(m-1)^+}\kap_3 \Bigl(t^m + \sum_{i\in\calI} \sup_{0\leq s\leq t}|W_i(s)|^m\Bigr)\nonumber
\\
&\leq  c_\ell (d+2)^{(m-1)^+} d^{m/2}(1+ \abs{x}^m)
+ c_\ell (d+2)^{(m-1)^+}\kap_3 \Bigl(t^m + \sum_{i\in\calI} \sup_{0\leq s\leq t}|W_i(s)|^m\Bigr).
\end{align}
Now we can find constants $\kap_i, i=4, 5, 6$, depending only on $m$, so that
\begin{align*}
\Exp_x \Bigl[\int_0^\infty e^{-\beta t}(t^m+ \sum_{i\in\calI} \sup_{0\leq s\leq t}|W_i(s)|^m)dt\Bigr]
&\leq \kap_4\Big(1+ \int_0^\infty e^{-\beta t}\Exp\bigl[\sup_{0\leq s\leq t}|W_1(s)|^m\bigr]\Big)ds
\\
&\leq \kap_5\Big(1+ \int_0^\infty e^{-\beta t}t^{m/2}\Big)ds
\\
& \leq \kap_6,
\end{align*}
where in the second inequality we use Burkholder-Davis-Gundy inequality \cite{kallenberg}. Therefore combining the
above display with \eqref{9} we obtain
\begin{equation}\label{10}
\Exp \Bigl[\int_0^\infty e^{-\beta t}\ell(X(t))dt\Bigr]\leq\kap_7(1 + \abs{x}^m),
\end{equation}
for some positive constant $\kap_7$. Let $\abs{h_i}_\infty\df \sup_{M\in\M}|h_i(M)|$ and $\abs{b}_1=\sum_i\sup_{x\in\Rd_+}|b_i(x)|$.
Then for each $i\in\{1, \ldots, d\}$,
\begin{align}\label{11}
\Big|\Exp_x \Bigl[\int_0^\infty e^{-\beta t}h_i(v(t))dY_i(t)\Bigr]\Big| & \leq
\abs{h_i}_\infty\Big|\Exp_x \Bigl[\int_0^\infty e^{-\beta t}dY_i(t)\Bigr]\Big|\nonumber
\\
&= \abs{h_i}_\infty\, \Big|\Exp_x \Bigl[\int_0^\infty \beta e^{-\beta t}Y_i(t)dt\Bigr]\Big|\nonumber
\\
&\leq \kap_8 \frac{\abs{h_i}_\infty}{1-\max_i\al_i} \,
\Big|\Exp \Bigl[\int_0^\infty \beta e^{-\beta t}(\abs{b}_1t+\sum_{i\in\calI}\sup_{0\leq s\leq t}|W_i(s)|)dt\Bigr]\Big|\nonumber
\\
&\leq \kap_9,
\end{align}
for some constants $\kap_8, \kap_9$, where the second equality is obtained by a use of a integration-by-parts formula
(this is justified by \eqref{a-3} and Law of Iterated Logarithm which confirms that maximum of a standard Brownian motion has polynomial growth almost surely), in the
third inequality we use \eqref{a-3}, and last inequality can be obtained following similar estimate as above (display above \eqref{10}).
We also note that the constants above do not depend on $i$ and $v\in\Uadm$. Therefore we complete the proof
combining \eqref{10} and \eqref{11}.
\end{proof}

Next lemma establishes continuity of the value function $V$.
\begin{lemma}\label{lem-cont}
The value function $V$ given by \eqref{3} is continuous in $\Rd_+$. In particular, $V\in\calC_{\mathrm{pol}}(\Rd_+)$.
\end{lemma}

\begin{proof}
To show the continuity we prove the following: for any sequence $\{v^n\}\in\Uadm$ and $x_n\to x\in \Rd_+$, as $n\to\infty$, we have
\begin{equation}\label{13}
|J(x_n, v^n)- J(x, v^n)|\to 0 , \quad \text{as}\quad n\to\infty,
\end{equation}
where $J$ is given by \eqref{2.5}. From the proof of Lemma~\ref{lem-growth} it is clear that
$$\limsup_{T\to\infty}\sup_{v\in\Uadm}\Exp_x\Bigl[ \int_T^\infty e^{-\beta s}\big(\ell(X(s))ds + \sum_{i\in\calI} h_i(v(s))dY_i(s)\Bigr]=0.$$
Therefore to prove \eqref{13} it is enough to show that for any $T>0$,
\begin{align}\label{14}
& \Exp_x\Bigl[ \int_0^T e^{-\beta s}\big(\ell(X^n(s))ds + \sum_{i\in\calI} h_i(v^n(s))dY^n_i(s)\Bigr]\nonumber\\
&\qquad-\Exp_x\Bigl[ \int_0^T e^{-\beta s}\big(\ell(\tilde{X}^n(s))ds + \sum_{i\in\calI} h_i(v^n(s))d\tilde{Y}^n_i(s)\Bigr]
\to 0,
\end{align}
as $n\to \infty$ where $(X^n, Y^n), (\tilde{X}^n, \tilde{Y}^n)$ are the solutions of \eqref{1} with the control $v^n$ and the initial data $x_n, x$, respectively. It should be noted that both the solutions
$(X^n, Y^n), (\tilde{X}^n, \tilde{Y}^n)$ are governed by the control $v^n$ for all $n$.
Also from \eqref{9} and the calculations in \eqref{11} we see that dominated convergence
theorem can be applied to establish \eqref{14} if we can show the following hold,
\begin{equation}\label{15}
\begin{split}
\Big| \int_0^T e^{-\beta s}\big(\ell(X^n(s))ds-\int_0^T e^{-\beta s}\big(\ell(\tilde{X}^n(s))ds \Big| &\to 0 \quad{a.s.,}
\\
\Big|\int_0^T e^{-\beta s}\sum_{i\in\calI} h_i(v^n(s))dY^n_i(s)- \int_0^T\sum_{i\in\calI} e^{-\beta s}h_i(v^n(s))d\tilde{Y}^n_i(s) \Big| & \to 0 \quad{a.s.}
\end{split}
\end{equation}
We recall the maps $(\calT, \calS)$ from Theorem~\ref{Thm-existence}. We have
\begin{align*}
Y^n_i(t)=\calT_i(X^n, Y^n)(t) & = \sup_{0\leq s\leq t} \max\{0, -U_i(X^n, Y^n)(s)\},
\\
X^n_i(t)=\calS_i(X^n, Y^n)(t) &= U_i(X^n, Y^n)(t) + \calT_i(X^n, Y^n)(t),
\\
\text{where} \;\; U_i(X^n, Y^n)(t) &=(x_n)_i +\int_0^t b_i(X^n(s))ds + (\Sigma W)_i(t) +\sum_{j\neq i}\int_0^tv^n_{ij}(s)dY^n_j(s).
\end{align*}
The above relation also holds if we replace $(X^n, Y^n), x_n$ by $(\tilde{X}^n, \tilde{Y}^n)$ and $x$, respectively. Let $\bvnorm{\cdot}_{st}$ denote the
bounded variation norm on the interval $[s, t]$. Then using the result from \cite[pp. 171]{shashiasvili} we have for all $i\in \calI$, 
\begin{equation*}
\begin{split}
\bvnorm{Y^n_i-\tilde{Y}^n_i}_{s t} &\leq \babs{Y^n_i(s)-\tilde{Y}^n_i(s)} + \babs{U_i(X^n, Y^n)(s)-U_i(\tilde{X}^n, \tilde{Y}^n)(s)}
\\
&\, \quad +\bvnorm{U_i(X^n, Y^n)-U_i(\tilde{X}^n, \tilde{Y}^n)}_{st}.
\end{split}
\end{equation*}
Now one can follow similar calculation as in \eqref{a-4} to obtain
\begin{align*}
\bvnorm{Y^n-\tilde{Y}^n}_{st} &\leq \sum_{i\in\calI}\babs{Y^n_i(s)-\tilde{Y}^n_i(s)} + 
sKd\norm{X^n-\tilde{X}^n}_s + Kd(t-s) \sup_{r\in[s, t]}\norm{X^n(r)-\tilde{X}^n(r)}
\\
&\quad +\max_{i}\al_i\, \bvnorm{Y^n-\tilde{Y}^n}_{st},
\end{align*}
which gives
\begin{equation}\label{16}
\bvnorm{Y^n-\tilde{Y}^n}_{st} \leq \frac{1}{1-\max_i\al_i}\Big(\sum_{i\in\calI}\babs{Y^n_i(s)-\tilde{Y}^n_i(s)} + Kds\norm{X^n-\tilde{X}^n}_s +
 Kd(t-s) \sup_{r\in[s, t]}\norm{X^n(r)-\tilde{X}^n(r)}\Big).
\end{equation}
Similarly, we get
\begin{align*}
\sup_{r\in[s, t]}\norm{X^n(r)-\tilde{X}^n(r)}\leq \norm{X^n(s)-\tilde{X}^n(s)} + Kd(t-s)\sup_{r\in[s, t]}\norm{X^n(r)-\tilde{X}^n(r)}
\\
+(1+ \max_i \al_i )\bvnorm{Y^n-\tilde{Y}^n}_{st}  + \sum_{i\in\calI}\babs{Y^n_i(s)-\tilde{Y}^n_i(s)}
\end{align*}
and using \eqref{16}  we have
\begin{align}\label{17}
\sup_{r\in[s, t]}\norm{X^n(r)-\tilde{X}^n(r)} & \leq 
\big(1+\frac{2\vee(2Kds)}{1-\max_i \al_i}\big)(\norm{X^n(s)-\tilde{X}^n(s)} + 
\sum_{i\in\calI}\babs{Y^n_i(s)-\tilde{Y}^n_i(s)})
\nonumber
\\
&\quad +Kd(t-s)(1+\frac{2}{1-\max_i \al_i})\sup_{r\in[s, t]}\norm{X^n(r)-\tilde{X}^n(r)}.
\end{align}
Now we choose $\theta>0$ so that $Kd\theta(1+\frac{2}{1-\max_i \al_i})<1$. Since $x_n\to x$, we obtain from \eqref{16} and  \eqref{17} that
$\norm{X^n-\tilde{X}^n}_\theta, \bvnorm{Y^n-\tilde{Y}^n}_\theta$ tends to $0$ as $n\to\infty$. Similar convergence can be extended on interval
$[0, k\theta], k$ is a positive integer,  by an iterative procedure and using \eqref{16} and \eqref{17}.
This proves \eqref{15}.
\end{proof}

Now we prove the dynamic programming principle for $V$. Recall that $\B_r(x)$ denotes the ball of
radius $r$ around $x$ in $\Rd_+$.
Let $(X, Y)$ be a solution of \eqref{1} with initial data $x$. Denote
\begin{equation}\label{20}
\sigma_r\df \inf\{t\geq 0\, :\, X(t)\notin \B_r(x)\}.
\end{equation}
That is, $\sigma_r$ denotes the exit time of $X$ from the ball $\B_r(x)$.

\begin{proposition}[Dynamic programming principle]\label{prop-dpp}
Let $\sigma=\sigma_r$ be as above. Then for any $t>0$ we have
\begin{equation}\label{12}
V(x) = \inf_{v\in\Uadm}\Exp_x\Bigl[\int_0^{\sigma\wedge t} e^{-\beta s}\big(\ell(X(s))ds + \sum_{i\in\calI} h_i(v(s))dY_i(s)\big)
 + e^{-\beta\, \sigma\wedge t}V(X(\sigma\wedge t))\Bigr].
\end{equation}
\end{proposition}

\begin{proof}
Using the Markov property of Brownian motion one gets that for any $v\in\Uadm$,
\begin{equation}\label{18}
 \Exp_x\Big[\int_{\sigma\wedge t}^\infty e^{-\beta\, s}(\ell(X(s))ds +\sum_{i\in\calI} h_i({v})dY_i)\, \Big| \mathcal{F}_{\sigma\wedge t}\Big]
\geq  e^{-\beta\, \sigma\wedge t} V(X(\sigma\wedge t)).
\end{equation}
See for example \cite[Lemma~3.2]{yong-zhou}. For $\eps>0$, we have $v\in\Uadm$ such that
\begin{align*}
V(x) & \geq J(x, v) - \eps
\\
&= \Exp\Big[ \int_0^{\sigma\wedge t}e^{-\beta\, s}(\ell(X(s))ds +\sum_{i\in\calI} h_i({v})dY_i) +
\\
&\, \quad\hspace{.2in}
\Exp_x\Big[\int_{\sigma\wedge t}^\infty e^{-\beta\, s}(\ell(X(s))ds +\sum_{i\in\calI} h_i({v})dY_i)\, \Big| \mathcal{F}_{\sigma\wedge t}\Big]\Big]-\eps
\\
&\geq  \Exp_x\Big[ \int_0^{\sigma\wedge t}e^{-\beta\, s}(\ell(X(s))ds +\sum_{i\in\calI} h_i({v})dY_i) +
e^{-\beta\, \sigma\wedge t} V(X(\sigma\wedge t))\Big]-\epsilon,
\end{align*}
where in the last inequality we use \eqref{18}. Since $\eps$ is arbitrary,  we get
\begin{equation}\label{19}
V(x)\geq \inf_{v\in\Uadm}\Exp_x\Big[ \int_0^{\sigma\wedge t}e^{-\beta\, s}(\ell(X(s))ds +\sum_{i\in\calI} h_i({v})dY_i) +
e^{-\beta\, \sigma\wedge t} V(X(\sigma\wedge t))\Big].
\end{equation}
To prove the other direction we use Lemma~\ref{lem-A1} from Appendix which says that
for given $\eps>0, \, v\in\Uadm$, there exist control $\tilde{v}\in\Uadm$ such that
$$\Exp_x\Big[\int_{\sigma\wedge t}^\infty e^{-\beta\, s}(\ell(X(s))ds +\sum_{i\in\calI} h_i(\tilde{v})dY_i)\, \Big| \mathcal{F}_{\sigma\wedge t}\Big]
\leq  e^{-\beta\, \sigma\wedge t} V(X(\sigma\wedge t))+ 3\eps,$$
and
$$v(s)=\tilde{v}(s)\quad \text{for all}\,\; s\leq \sigma\wedge t.$$
Then
\begin{align*}
V(x) &\leq \Exp_x\Big[ \int_0^{\sigma\wedge t}e^{-\beta\, s}(\ell(X(s))ds +\sum_{i\in\calI} h_i({v})dY_i) +
\\
&\hspace{.2in}  \Exp_x\Big[\int_{\sigma\wedge t}^\infty e^{-\beta\, s}(\ell(X(s))ds +\sum_{i\in\calI} h_i(\tilde{v})dY_i)\, \Big| \mathcal{F}_{\sigma\wedge t}\Big]\Big]
\\
&\leq \Exp_x\Big[ \int_0^{\sigma\wedge t}e^{-\beta\, s}(\ell(X(s))ds +\sum_{i\in\calI} h_i({v})dY_i)
+ e^{-\beta\, \sigma\wedge t} V(X(\sigma\wedge t)) \Big]
+ 3\eps.
\end{align*}
Now we take the infimum over $v$ to get the other side inequality. Hence \eqref{12} follows using \eqref{19} and the above display.
\end{proof}

\begin{lemma}\label{lem-small}
Given $r, \eps\in(0, 1)$ there exists $t\in(0, 1)$ such that
$$\sup_{v\in\Uadm}\Prob_x(\sigma_r\leq t)<\eps \quad \text{for all}\, \, x\in\Rd_+,$$
where $\sigma_r$ is given by \eqref{20}.
\end{lemma}

\begin{proof}
Fix $x\in\Rd$ and $v\in\Uadm$. Let $(X, Y)$ be the solution of \eqref{1} given the  control $v$ and initial data $x$. Then for
$i\in\calI$ we have
\begin{equation}\label{21}
X_i(t) = x_i + \int_0^t b_i(X(s))ds + (\Sigma W(t))_i + \sum_{j=1}^d \int_0^t v_{ij}(s) dY_j(s), \quad t\geq 0.
\end{equation}
Then summing over $i$ we get from \eqref{21} that
\begin{align}\label{22}
\sum_{i\in\calI}\abs{X_i(t)-x_i} &\leq t \sum_{i\in\calI} \sup_{ x\in\Rd_+}\abs{b_i(x)} + \sum_{i\in\calI} \babs{(\Sigma W(t))_i}
+ (1+\max_i\al_i)\sum_{i\in\calI} Y_j(t)\nonumber
\\
&\leq \kap_1 \Bigl(t \sup_{x\in\Rd_+}\abs{b(x)} + \sum_{i\in\calI}\sup_{s\leq t} \abs{W_i(s)}\Bigr),
\end{align}
for some constant $\kap_1$, depending on $\Sigma, \al$, where in the last inequality we use \eqref{a-3}.
Now if $\abs{X(t)-x}\geq r$ and $t\, \kap_1\sup_{x\in\Rd_+}\abs{b(x)}<r/2$ then using \eqref{22} we obtain
\begin{align*}
r/2<  \kap_1\, \sum_{i\in\calI}\sup_{s\leq t} \abs{W_i(s)}.
\end{align*}
Thus we have for the above choice $t$ that
\begin{align*}
\sup_{v\in\Uadm}\Prob_x\Bigl(\sup_{s\leq t}\abs{X(s)-x}\geq r\Bigr) &\leq \Prob\Bigl(\sum_{i\in\calI}\sup_{s\leq t} \abs{W_i(s)}>\frac{r}{2\kap_1}\Bigr)
\\
&\leq \frac{4d^2\kap^2_1}{r^2} \Exp \bigl[\sup_{s\leq t} |W_1(s)|^2 \bigr]
\\
&\leq \frac{16d^2\kap^2_1}{r^2}\, t,
\end{align*}
where in the last line we use Doob's martingale inequality. Thus we can further restrict $t$, depending on $\eps$, so that the rhs of above
display is smaller that $\eps$. This completes the proof observing that $\Prob_x(\sigma_r\leq t)\leq \Prob_x(\sup_{s\leq t}\abs{X(s)-x}\geq r)$.
\end{proof}

Now we are ready to show that the value function $V$, given by \eqref{3}, is a viscosity solution. This clearly 
concludes the proof of Theorem \ref{Thm-main} (a). 

\begin{theorem}\label{Thm-viscosity}
Assume that Conditions~\ref{cond1}--\ref{cond3} hold.
The value function $V$ is a viscosity solution of the HJB given by \eqref{4}.
\end{theorem}

\begin{proof} 
First we show that $V$ is a super-solution.
Let $\varphi\in\calC^2(\Rd_+)$ be such that $V-\varphi$ attends its local minimum at $x\in\Rd_+$. We can also assume that $V(x)=\varphi(x)$,
otherwise we have to translate $\varphi$. Therefore we need to show that
\begin{equation}\label{23}
(\calL \varphi(x) -\beta\varphi (x)+ \ell(x)) \wedge \min_{x\in\I(x)}\calH_i(D\varphi(x))\leq 0.
\end{equation}
Let us assume that \eqref{23} does not hold and we have a $\theta>0$ such that
$$(\calL \varphi(x) -\beta\varphi (x)+ \ell(x)) \wedge \min_{x\in\I(x)}\calH_i(D\varphi(x))\geq 2\theta.$$
Using the continuity property we can find $r>0$ such that the following holds
\begin{equation}\label{24}
\begin{split}
\calL \varphi(y) - \beta \varphi(y) + \ell(y) & \geq \theta 
\quad \text{for all}\, \; y\in\B_r(x),\;\; V-\varphi\geq 0 \;\text{on}\;\; \B_r(x),
\\
\calH_i(D\varphi(y)) &\geq \theta \quad \text{for}\;\; y\in\B_r(x), \; \text{and } \; \I(y)\subset\I(x)  \;\; \text{for}\;\; y\in\B_r(x).
\end{split}
\end{equation}
From \eqref{24} we get
\begin{equation}\label{25}
D\varphi(y)\cdot Me_i + h_i(M)\geq \theta  \quad \text{for all} \;\; M\in\M,\;\;i\in\I(y), \; y\in\B_r(x).
\end{equation}
Let $v\in\Uadm$ and $(X, Y)$ be the corresponding solution to \eqref{1} with initial data $x$. Define
$\sigma=\sigma_r=\inf\{t\geq 0\; :\; X(t)\notin \B_r(x)\}$.
Apply It\^{o} formula  \cite{kallenberg} to obtain
\begin{align*}
V(x)=\varphi(x)& = \Exp_x[e^{-\beta\, \sigma\wedge t}\varphi(X(\sigma\wedge t))]
\\
&\, \hspace{.2in}-\Exp_x\Big[\int_0^{\sigma\wedge t}e^{-\beta s}\bigl(\calL\varphi(X(s))ds -\beta\varphi(X(s))ds + D\varphi(s)\cdot v(s)dY(s)\bigr)\Big]
\\
&\leq \Exp_x[e^{-\beta\, \sigma\wedge t}V(X(\sigma\wedge t))]
\\
&\, \hspace{.2in}+\Exp_x\Big[\int_0^{\sigma\wedge t}e^{-\beta s}\bigl(\ell(X(s))ds -\theta ds- D\varphi(s)\cdot v(s)dY(s)\bigr)\Big],
\end{align*}
where in the last line we use \eqref{24}. Again using \eqref{25} we see that
\begin{align*}
\int_0^{\sigma\wedge t} e^{-\beta\, s}D\varphi(X(s))\cdot v(s)dY(s) &= \sum_{i\in\calI} \int_0^{\sigma\wedge t} e^{-\beta\,s}D\varphi(X(s))\cdot v(s)e_i\, dY_i(s)
\\
&\geq \sum_{i\in\calI}\int_0^{\sigma\wedge t} e^{-\beta\, s} \big(\theta- h_i(v(s))\big) dY_i(s) ,
\end{align*}
where we use the fact that $\int_0^{\sigma\wedge t}\sum_{i\notin\I(X(s))}dY_i(s) =0$. Thus combining above two displays we obtain
that for any  $v\in\Uadm$,
\begin{align*}
V(x) &\leq \Exp_x\Bigl[e^{-\beta\, \sigma\wedge t}V(X(\sigma\wedge t)) +\int_0^{\sigma\wedge t}e^{-\beta s}\bigl(\ell(X(s))ds +\sum_{i\in\calI}
h_i(v(s))dY_i(s)\bigr)\Bigr]
\\
&\, \hspace{.2in} - \Exp_x\Bigl[\int_0^{\sigma\wedge t} e^{-\beta s} \big(\theta ds + \theta \sum_{i\in\calI} dY_i(s) \big)\Bigr].
\end{align*}
Since $v\in\Uadm$ is arbitrary we have from above that
\begin{equation}\label{26}
V(x) \leq \inf_{v\in\Uadm}\Exp_x\Bigl[e^{-\beta\, \sigma\wedge t}V(X(\sigma\wedge t)) +\int_0^{\sigma\wedge t}e^{-\beta s}\bigl(\ell(X(s))ds +\sum_{i\in\calI}
h_i(v(s))dY_i(s)\bigr)\Bigr]-\al(t),
\end{equation}
where $$\al(t)=\inf_{v\in\Uadm}\Exp_x\Bigl[\int_0^{\sigma\wedge t} e^{-\beta s} \big(\theta ds + \theta \sum_{i\in\calI} dY_i(s) \big)\Bigr].$$
In view Lemma~\ref{lem-small} we can find $t$ so that $\al(t)>0$ which will give us from \eqref{26}
\begin{equation*}
V(x) < \inf_{v\in\Uadm}\Exp_x\Bigl[e^{-\beta\, \sigma\wedge t}V(X(\sigma\wedge t)) +\int_0^{\sigma\wedge t}e^{-\beta s}\bigl(\ell(X(s))ds +\sum_{i\in\calI}
h_i(v(s))dY_i(s)\bigr)\Bigr],
\end{equation*}
but this is contradicting to Proposition~\ref{prop-dpp}. This establishes \eqref{23}.

Now we show that $V$ is also a subsolution. Let $\varphi\in\calC^2(\Rd)$ be such that $V-\varphi$ attends a local maximum at the point $x\in\Rd_+$.
Also assume that $V(x)=\varphi(x)$.
We have to show that
\begin{equation}\label{27}
(\calL \varphi(x) -\beta\varphi (x)+ \ell(x)) \vee \max_{i\in\I(x)}\calH_i(D\varphi(x))\geq 0.
\end{equation}
If not, then  we can find $\theta>0$ such that
\begin{equation}\label{28}
(\calL \varphi(x) -\beta\varphi (x)+ \ell(x)) \vee \max_{i\in\I(x)}\calH_i(D\varphi(x))\leq -2\theta.
\end{equation}
By Condition ~\ref{cond3} there exists $M\in\M$ such that
$$\max_{i\in\I(x)}\{D\varphi(x)\cdot Me_i + h_i(M)\}< -\theta. $$
Therefore using \eqref{28} we can find $r>0$ such that $V(y)-\varphi(y)\leq 0$ for $y\in\B_r(x)$ and the following hold:
\begin{equation}\label{29}
\begin{split}
\calL \varphi(y) -\beta\varphi (y)+ \ell(y) &<-\theta/2  \quad \text{for}\, \; y\in \B_r(x),
\\
\{D\varphi(y)\cdot Me_i + h_i(M)\} & < -\theta/2  \quad \text{for}\,\, i\in\I(y), \;\; \I(y)\subset \I(x),\; \text{and}\; y\in\B_r(x).
\end{split}
\end{equation}
Now we take the control $v\equiv M$ and follow the same calculations as above, and using \eqref{29}, to arrive at
\begin{align*}
V(x) &\geq \Exp_x\Bigl[e^{-\beta\, \sigma\wedge t}V(X(\sigma\wedge t)) +\int_0^{\sigma\wedge t}e^{-\beta s}\bigl(\ell(X(s))ds +\sum_{i\in\calI}
h_i(v(s))dY_i(s)\bigr)\Bigr]
\\
&\, \hspace{.2in} + \Exp_x\Bigl[\int_0^{\sigma\wedge t} e^{-\beta s} \big(\theta/2 ds + \theta/2 \sum_{i\in\calI} dY_i(s) \big)\Bigr].
\end{align*}
Now using Lemma~\ref{lem-small} we know that the last term on rhs of above display is strictly positive for a suitably chosen $t$ and therefore we obtain
\begin{align*}
V(x) > \Exp_x\Bigl[e^{-\beta\, \sigma\wedge t}V(X(\sigma\wedge t)) +\int_0^{\sigma\wedge t}e^{-\beta s}\bigl(\ell(X(s))ds +\sum_{i\in\calI}
h_i(v(s))dY_i(s)\bigr)\Bigr],
\end{align*}
but this is contradicting to Proposition \ref{prop-dpp}. This proves \eqref{27}. Hence $V$ is a viscosity solution to \eqref{4}.
\end{proof}

\section{Proof of Theorem~\ref{Thm-main}(b) }\label{S-uni}

In this section we show the uniqueness property of the viscosity solutions of \eqref{4}. Indeed, we prove the following theorem. 

\begin{theorem}\label{T-comp} Let Conditions~\ref{cond1} and \ref{cond2} hold.
Assume that there exists a vector $c=(c_1,\ldots,c_d)\in\R^d$ 
such that $h_i(M)=c\cdot Me_i$ for all $i\in\calI$. 
If $u\in\calC_{\mathrm{pol}}(\R^d_+)$ and $v\in\calC_{\mathrm{pol}}(\R^d_+)$ are, 
respectively, a viscosity subsolution and a viscosity supersolution to 
\eqref{4}, then $u\leq v$ on $\R^d_+$.
\end{theorem}

Theorem \ref{Thm-main} (b) follows immediately from the theorem above.

\begin{proof}[Proof of Theorem \ref{Thm-main} (b)] 
According to Theorem \ref{Thm-main} (a) (or Theorem \ref{Thm-viscosity}) 
and Lemma \ref{lem-growth} assure that the value function $V$ given by 
\eqref{3} is a viscosity solution to \eqref{4} 
in the class $\calC_{\mathrm{pol}}(\R^d_+)$. Theorem \ref{T-comp} then guarantees 
that $V$ is the unique viscosity solution to \eqref{4} in $\calC_{\mathrm{pol}}(\R^d_+)$. 
\end{proof}

To establish the comparison theorem above, it is crucial to find a test function 
having a few of properties, stated in the next theorem.

\def\gl{\lam}

\begin{theorem}\label{test-f} Assume that $h_i(M)\equiv 0$ for all $i\in\calI$. 
Then 
there exists a globally $\calC^{1,1}$ function $\ph$ on $\R^d$ 
having the properties:  
\begin{align*}
&\ph(y)>0\ \text{ for all }\ y\in \R^d\setminus\{ 0\}, \\
&\ph(\gl y)=\gl^2 \ph(y) \ \text{ for all }\ y\in\R^d \ \text{ and }\ \gl\geq 0,
\intertext{and for each $i\in\calI$,} 
&\calH_i(D\ph(x))\begin{cases}
\geq 0 &\text{ if }\ x_i\geq 0,\\
\leq 0 &\text{ if }\ x_i\leq 0.
\end{cases} \\
\end{align*}
\end{theorem}

We note that if $d>1$ and $h_i=0$ for some $i\in\calI$, then 
\[
\calH_i(p)=\min_{M\in\M(\al)}p\cdot M e_i=p_i-\al_i\max_{j\in\calI\setminus\{i\}}p_j^+,
\]
while if $d=1$ and $h_1=0$, then 
\[
\calH_1(p)=p. 
\]

The proof of the theorem above is given in the next section. 

\begin{proof}[Proof of Theorem \ref{T-comp}]
Let $u\in\calC_{\mathrm{pol}}(\R^d_+)$ and $v\in\calC_{\mathrm{pol}}(\R^d_+)$ be,  
respectively, a viscosity subsolution and a viscosity supersolution to 
\eqref{4}.  

First of all, we note that it can be assumed that 
$c=0$ and $\beta=1$. Indeed, setting $u_1(x)=u(x)+c\cdot x$ and $v_1(x)=v(x)+c\cdot x$, 
observing that $w:=u_1$ (resp., $w:=v_1$) is 
a viscosity subsolution (resp., supersolution) to the HJB
\[
\begin{cases}
\calL w -\beta w + \ell(x) +c\cdot(-b(x)+\beta x) = 0\quad  \text{ in }\, (\R^d_+)^\circ,&\\[3pt]
\displaystyle 
\min_{M\in\M(\al)}Dw\cdot M e_i = 0 \quad \text{ for }\, x
\in\partial\R^d_+,\ i\in\I(x).&
\end{cases}
\]
This boundary value problem can be rewritten as
\[
\begin{cases}
\beta^{-1}\calL w -w + \beta^{-1}\big(\ell(x) -c\cdot b(x)\big)+c\cdot x  = 0\quad  \text{ in }\, (\R^d_+)^\circ,&\\[3pt]
\displaystyle 
\min_{M\in\M(\al)}Dw\cdot M e_i = 0 \quad \text{ for }\, x\in
\partial\R^d_+,\ i\in\I(x).&
\end{cases}
\]
Thus, replacing $u$, $v$, $A$, $b$ and $\ell$ by 
$u_1$, $v_1$, 
$\beta^{-1}A$, $\beta ^{-1}b$ 
and $\beta^{-1}(\ell-c\cdot b)+c\cdot x$, if necessary, we may assume that 
$c=0$ and  $\beta=1$. 
\def\N{\mathbb N}

Henceforth we assume that $c=0$ and $\beta=1$. In particular, we have 
$h_i=0$ for all $i\in\calI$.   
Since $u,v\in\calC_{\mathrm{pol}}(\R^d_+)$, there are constants $m\in\N$ and $\kappa>0$ such that
\[
|u(x)|\vee|v(x)|\leq \kappa(1+|x|)^m \ \ \text{ for all }\ x\in\R^d_+. 
\] 
Let $\eps>0$ and set for $x\in\R^d_+$,  
\[\begin{gathered}
\psi(x)=(m+1)^{-1}(x_1^{m+1}+\cdots+x_d^{m+1})+\Ir\cdot x, \\
u_\eps(x)=u(x)-\eps(C+\psi(x)) \ 
\ \text{ and }\ \ v_\eps(x)=v(x)+\eps(C+\psi(x)),  
\end{gathered}
\]
where $C$ is a constant to be fixed soon. 

Note that $\psi\in \calC^2(\R^d_+)$,  
\begin{equation}\label{comp-4}
\lim_{|x|\to\infty} u_\eps(x)=-\infty \ \ \text{ and } \ \ 
\lim_{|x|\to\infty}v_\eps(x)=\infty,
\end{equation}
and observe that for some constants $C_i>0$, with $i=1,2,3$, and  for all 
$x\in\R^d_+$,
\[
|D\psi(x)|\leq C_1|x|^{m}, \ \ 
|\calL \psi(x)|\leq C_2(1+|x|)^{m}  \ \ 
\text{ and } \ \ 
|\calL \psi(x)|\leq C_3+\psi(x), 
\]
where we have used the fact that $b$ is bounded on $\R^d_+$.  

Now, we fix $C=C_3$. Let $x\in\R^d_+$ and 
$(p,X)\in J^{2,+}u_\eps(x)$. See \cite{crandall-ishii-lions} 
for the definition of semi-jets $J^{2,\pm}$.  
Setting $(q,Y):=(p,X)+\eps(D\psi(x),D^2\psi(x))\in J^{2,+}u(x)$, 
we compute that  
\begin{equation}\label{comp-1}
\begin{aligned}
\Tr AX&+b(x)\cdot p-u_\eps(x)+\ell(x) 
\\&\geq \Tr AY+b(x)\cdot q-u(x)+\ell(x)-\eps|\calL \psi(x)|+\eps (C+\psi(x))
\\&\geq \Tr AY+b(x)\cdot q-u(x)+\ell(x). 
\end{aligned}
\end{equation} 
Recalling that 
\begin{align*}\calH_i(p) = p_i - \al_i \max_{j\neq i}\, p_j^+ \quad 
\text{ for all }\, i\in\calI,
\end{align*} 
and noting that for any $i\in\calI$, 
\[
D_i\psi(x)=x_i^{m}-1 \ \ \text{ for all }\ x\in\R^d_+,
\] 
where $D_i=\partial/\partial x_i$, we see that if $x_i=0$, then   
\[\begin{aligned}
\calH_i(p) &=  q_i +\eps- 
\al_i \max_{j\neq i}\,\big(q_j +\eps (-x_j^m+1)\big)^+
\\
&\geq q_i+\eps(1-\al_i) - \al_i \max_{j\neq i} q_j^+
=\calH_i(q)+\eps(1-\al_i).
\end{aligned}
\]
We thus deduce that $w:=u_\eps$ is a viscosity subsolution to 
\begin{equation}\label{comp-2}
\begin{cases}
\calL w-w+\ell=0 \ \ \text{ in } \ (\R^d_+)^\circ, &\\[3pt]
\calH_i^-(Dw(x))=0 \ \ \text{ if }\ x \in \partial \R^d_+, \ i\in\I(x),  
\end{cases}
\end{equation}
where $\calH_i^-(p):=\calH_i(p)-\eps(1-\al_i)$. 

Similarly, we see that $w:=v_\eps$ is a viscosity supersolution to 
\begin{equation}\label{comp-3}
\begin{cases}
\calL w-w+\ell=0 \ \ \text{ in } \ (\R^d_+)^\circ, &\\[3pt]
\calH_i^+(Dw(x))=0 \ \ \text{ if }\ x \in \partial \R^d_+, \ i\in\I(x),  
\end{cases}
\end{equation}
where $\calH_i^+(p):=\calH_i(p)+\eps(1-\al_i)$.

It is enough to show that for any $\eps>0$, we have $u_\eps\leq v_\eps$ 
on $\R^d_+$.

To do this, we fix $\eps>0$ and suppose to the contrary that 
$\sup{(u_\eps-v_\eps)}>0$.  
Let $\ph\in \calC^{1,1}(\R^d)$ be the test function given by Theorem 
\ref{test-f}. For any $k\in\N$, consider the function 
\[
\Phi(x,y)=u_\eps(x)-v_\eps(y)-k\ph(x-y) \ \ \text{ on } (x,y)\in\left(\R^d_+\right)^2,
\]
and let $(x_k,y_k)$ be a maximum point of this function $\Phi$. Note that 
the existence of such a maximum point is ensured by \eqref{comp-4}. 

\def\II{\mathbb{I}}

Since $\ph\geq 0$ is globally $\calC^{1,1}$  on $\R^d$, we find that 
for all $(x,y)\in \left(\R^d_+\right)^2$ and some constant $C>0$, 
\[
\ph(x-y)\leq \ph(x_k-y_k)+D\ph(x_k-y_k)\cdot(x-y-(x_k-y_k)) 
+C|x-y-(x_k-y_k)|^2.
\]
It is standard to see (see \cite{crandall-ishii-lions} for instance) 
that 
\begin{equation}\label{comp-5}
\lim_{k\to \infty} k\ph(x_k-y_k)=0 
\end{equation}
and that there are symmetric matrices $X_k,\,Y_k\in\mat$ such that 
\begin{equation}\label{comp-6}
\begin{pmatrix} X_k&0 \\ 0&-Y_k \end{pmatrix}
\leq 6Ck\begin{pmatrix}\II&-\II \\ -\II&\II \end{pmatrix},
\end{equation} 
and 
\begin{equation}\label{comp-7}
(p_k,X_k)\in\bar J^{2,+}u_\eps(x_k) \ \ \text{ and } \ \  
(p_k,Y_k)\in\bar J^{2,-}v_\eps(y_k),
\end{equation}
where $p_k=k D\ph(x_k-y_k)$ and $\bar J^{2,\pm}$ denotes the ``closure'' 
(see \cite{crandall-ishii-lions}) of the semi-jets 
$J^{2,\pm}$.  
The inequality above implies that $X_k\leq Y_k$. 

Let $x_k=(x_{k1},\ldots,x_{kd})$ and $y_k=(y_{k1},\ldots,y_{kd})$. 
If $x_{ki}=0$ for some $i\in\calI$, then $x_{ki}-y_{ki}\leq 0$ 
and hence
\[
\calH_i^-(p_k)=k\calH_i(D\ph(x_k-y_k))-\eps(1-\al_i)\leq 
-\eps(1-\al_i)<0. 
\]
Thus, by the viscosity property of $u_\eps$, we see that 
\[
\Tr AX_k +b(x_k)\cdot p_k -u_\eps(x_k)+\ell(x_k)\geq 0. 
\]
Similarly, we obtain 
\[
\Tr A Y_k+b(y_k)\cdot p_k-v_\eps(y_k)+\ell(y_k)\leq 0.
\]
Combining these two, we get 
\begin{equation}\label{comp-8}\begin{aligned}
u_\eps(x_k)-v_\eps(y_k)
&\leq \Tr A(X_k-Y_k)+(b(x_k)-b(y_k))\cdot p_k +\ell(x_k)-\ell(y_k)
\\&\leq K|x_k-y_k||p_k|+\ell(x_k)-\ell(y_k),  
\end{aligned}
\end{equation}
where $K$ is a Lipschitz bound of the function $b$. 
Furthermore, observe that $\{x_k,y_k\}\subset \big(\R^d_+\big)^2$ is a bounded 
sequence by \eqref{comp-4} and in view of the homogeneity of $\ph$ that 
for all $z\in\R^d$ and for some constant $C_4>0$,
\[
|z||D\ph(z)|\leq C_4\ph(z), 
\]
and, consequently, 
\[
|x_k-y_k||p_k|\leq C_4k\ph(x_k-y_k).
\]
Thus, from \eqref{comp-8} and \eqref{comp-5}, we infer that
\[\begin{aligned}
0&<\sup_{\R^d_+}(u_\eps-v_\eps)\leq u_\eps(x_k)-v_\eps(y_k)-k\ph(x_k-y_k)
\\&\leq K|x_k-y_k||p_k|+\ell(x_k)-\ell(y_k) \to 0 \ \ \text{ as } \ k\to \infty, 
\end{aligned}
\]
which is a contradiction, and conclude that $u_\eps\leq v_\eps$ on $\R^d_+$.  
\end{proof}

\begin{remark}
The above proof also works if we replace the operator $\calL V- \beta V + \ell$ by a general nonlinear operator
satisfying some standard condition. One may look at 
\cite[Conditions (0.1)-(0.3)]{crandall-ishii-lions},\cite[Condition~(1.3)]{dupuis-ishii-90}, \cite[Condition~(1.4)]{dupuis-ishii} for
general form of such operators.
\end{remark}

%
%
%

\section{Construction of a test function} \label{const-tf}
We now give the proof of Theorem \ref{test-f}. That is, we
construct a test function in the class $\calC^{1,1}(\R^d)$, which is 
a main step in establishing the comparison principle. 

When $d=1$, the function $\ph(x)=x^2$ obviously has the properties of test function
stated in Theorem \ref{test-f}.  We may thus assume in this section that $d>1$.

\subsection{A convex body} 
We set
$$
\Xi^d:=\{ {\xi}=(\xi_1,\ldots,\xi_d)\mid \xi_i=1\text{ or }\alpha_i \text{ for all }\, i\in\calI\}.
$$

For $ {\xi}=(\xi_1,\ldots,\xi_d)\in \Xi^d$, we set
\[
\xi^\#=\#\{i\in \calI\mid \xi_i=1\},
\]
where $\# A$ indicates the cardinality of the set $A$, and, for $k\in \{0,1,\ldots,d\}$,
\[
\Xi^d_k:=\{ {\xi}\in \Xi^d\mid \xi^\#=k\}.
\]
We set
\[
\Xi_*^d:=\bigcup_{k\in \calI} \Xi_k^d=\Xi^d\setminus\{(\al_1,\ldots,\al_d)\},
\]
and note
$$\#\Xi^d_k={d \choose k},\quad \#\Xi_*^d=2^d-1. $$
Let
\[
1=\theta_1<\cdots <\theta_d,
\]
and assume that
\begin{equation}\label{a-15}
\theta_1>\frac{d-1+\max_{i\in\calI}\al_i}{d}\,\theta_d.
\end{equation}
Note that using Condition~\ref{cond2} we have
\begin{equation}\label{a-16}
\theta_1>\frac{d-1+\max_{i\in\calI}\al_i}{d}\,\theta_d
>\frac{d\max_{i\in\calI}\al_i}{d}\,\theta_d
=\max_{i\in \calI}\al_i\,\theta_d.
\end{equation}
For $ {\xi}\in \Xi_k^d$, we define
\[
H( {\xi}):=\{{ {x}}\in\Rd\mid  {\xi}\cdot { {x}}\leq \theta_k\},
\]
and, for $\delta>0$, set
\[
H_\del^-:=\{{ {x}}\in\R^n\mid x_i\geq -\del \ \text{ for all } 1\leq i\leq d\},
\]
\[
\mathcal{S}:=\bigcap_{ {\xi}\in \Xi_*^d} H( {\xi}),
\]
and
\[
\mathcal{S}_\del:=\mathcal{S}\cap H_\del^-.
\]
Note that
\[
\partial H( {\xi})=\{{ {x}}\in\Rd\mid  {\xi}\cdot { {x}}=\theta_k\} \ \ \text{ for }  {\xi}\in \Xi_k^d.
\]

\begin{lemma} \label{lem-A3}
The set $\mathcal{S}_\del$
is a compact convex subset of $\Rd$ and is
a neighborhood of the origin.
\end{lemma}

\begin{proof} It is clear that $\mathcal{S}_\del$ is a closed convex subset of $\Rd$.
Observe that
\[
\mathcal{S}_\delta\subset H_\del^-\cap H(\mathbf{1}),
\]
and $H_\del^-\cap H(\mathbf{1})$ is bounded in $\Rd$.
Hence, $\mathcal{S}_\del$ is a compact subset of $\Rd$.

The distance from the origin to the hyperplane
\[
\partial H( {\xi})=\{{ {x}}\in\Rd\mid  {\xi}\cdot { {x}}=\theta_k\},
\]
with $ {\xi}\in \Xi_k^d$, is given by
\[
\theta_k/\abs{\xi}.
\]
Hence, 
setting
\[
\rho:=\del\wedge \min\{\theta_k/\abs{\xi}\mid  {\xi}\in \Xi_k^d,\, k\in\calI\}\ (\leq\del),
\]
we find that
\[
B_\rho(0)\subset \mathcal{S}_\del.
\]
Thus, $\mathcal{S}_\del$ is a neighborhood of the origin.
\end{proof}

Let ${ {z}}\in\partial \mathcal{S}_\del$ and let
$N({ {z}},\mathcal{S}_\del)$ denote the normal cone of $\mathcal{S}_\del$ at ${ {z}}$, that is, 
\begin{equation}\label{0000}
N(z,\mathcal{S}_\del):=\{ p\in\Rd\mid  
p\cdot({ {x}}-{ {z}})\leq 0 \ \text{ for all }{ {x}}\in \mathcal{S}_\del\}.
\end{equation}
Set
\[
J_z:=\{i\in\calI\mid z_i=-\del\},
\]
and
\[
\Xi_z:=\{ {\xi}\in \Xi_*^d\mid { {z}}\in\partial H( {\xi})\}.
\]

According to \cite[Corollary 23.8.1]{rock}, we have 

\begin{lemma} \label{lem-A4}
We have
\[
N({ {z}},\mathcal{S}_\del)=\mathrm{conical\ hull}\, \big(\Xi_z\cup \{-{ {e}_j}\mid j\in J_z \}\big)
=\sum_{j\in J_z}\R_+(-{ {e}_j})+\R_+\, \co \Xi_z,
\]
where $\co \Xi_z$ denotes the convex hull of $\Xi_z$.
\end{lemma}

For the time being we examine the normal cone $N({ {z}},\mathcal{S})$ at ${ {z}}$ for
${ {z}}\in\partial \mathcal{S}\cap \partial\R_+^d$.
Note that
\[\partial\R_+^d=\bigcup_{i\in \calI} F_i,\]
where $\,F_i:=\{{ {x}}\in\R_+^d\mid x_i=0\}$,
and therefore
\[
N({ {z}},\mathcal{S})=\R_+\,\co \Xi_z.
\]

\begin{lemma}\label{lem-A5}
Assume that ${ {z}}\in F_i\cap\partial \mathcal{S}$. For any $
 {\xi}\in \co \Xi_z$,
we have
\[
\xi_i=\al_i \ \ \text{ and } \ \ \max_{j\in\calI}\xi_j=1.
\]
\end{lemma}

\begin{proof} For notational convenience, we treat only the case when $i=d$.
Let $ {\xi}\in\co\Xi_z$. In order to show that $\xi_d=\al_d$, we suppose
to the contrary that $\xi_d\not=\al_d$. This implies that there exists
$\bar{ {\xi}}\in \Xi_z$ such that $\bar\xi_d=1$.  We select $k\in\calI$ so that
$\bar{ {\xi}}\in \Xi_k^d$.
Set
\[
\hat{ {\xi}}=(\bar\xi_1,\ldots,\bar\xi_{d-1},\al_d).
\]
First we consider the case when $k>1$. Note that $\hat{ {\xi}}\in \Xi_{k-1}^d$
and that
\[
\hat{ {\xi}}\cdot { {z}}\leq \theta_{k-1}\ \ \text{ and } \ \  \bar{ {\xi}}\cdot { {z}}=\theta_k.
\]
The last two relations are contradicting, since $\theta_{k-1}<\theta_k$ and
\[
\hat{ {\xi}}\cdot { {z}}=\hat{{\xi}}\cdot (z_1,\ldots,z_{n-1},0)
=\bar{ {\xi}}\cdot (z_1,\ldots,z_{n-1},0)=\bar{{ {\xi}}}\cdot { {z}}.
\]
Next consider the case when $k=1$. Note that
\[
\hat { {\xi}}=(\al_1,\ldots,\al_d)
\]
and
\[
\theta_1=\bar{ {\xi}}\cdot { {z}}=\hat { {\xi}}\cdot { {z}}\leq \max_{i\in\calI}\al_i\, \mathbf{1}\cdot { {z}}
\leq \max_{i\in\calI}\, \al_i\, \theta_d,
\]
which contradicts \eqref{a-16}.

Next we show that
\[
\max_{i\in\calI}\xi_i=1.
\]
We argue by contradiction, and suppose that
\[
\max_{i\in\calI}\xi_i<1.
\]
Since ${ {\xi}}\in \co \Xi_z$,
there exist $\xi^1,\ldots,\xi^N\in \Xi_z$
and $\lambda_i\geq 0$, with
\[
\sum_{i=1}^N\lam_i=1 \quad \text{such that} \quad { {\xi}}=\sum_{i=1}^N \lam_i{ {\xi}}^i,
\]
where we may assume that $N\geq d$. 

Writing
\[
{ {\xi}}^i=(\xi^i_1,\ldots,\xi^i_d) \ \ \ \text{ for }i\in \{1,\ldots,N\},
\]
by the supposition that
\[
\max_{i\in\calI} \xi_i<1,
\]
we find that for each $i\in\calI$, there exists $j_i\in\{1,\ldots,N\}$ such that
\[
\xi^{j_i}_i=\al_i.
\] 
We can rearrange $(\xi^1, \lam_1), \ldots, (\xi^N, \lam_N)$ in such a way 
that  $N\geq d$ and $j_i\leq d$ for all $i\in\calI$.
For $i\in\calI$, let
$k_i\in\calI$ be such that ${ {\xi}}^{i}\in \Xi_{k_i}^d$.
Observe that for any $i\in\calI$,
\[
{ {\xi}}^i\cdot { {z}}=\theta_{k_i},
\]
and hence
\[
\frac{1}{d}\sum_{i\in\calI}{ {\xi}}^i\cdot { {z}}=\frac{1}{d}\sum_{i\in\calI}
\theta_{k_i}\geq \theta_1.
\] 
Since ${{z}}\in \mathcal{S}\cap F_d\subset H(\mathbf{1})\cap F_d$, we have
\begin{align*}
\frac{1}{d}\sum_{i\in\calI}{ {\xi}}^i\cdot { {z}}
\leq \frac{(d-1)+\max \al_i}{d}\,\mathbf{1}\cdot { {z}} 
&\leq \frac{(d-1)+\max \al_i}{d}\,\theta_d.
\end{align*}
Thus combining above two displays, we get
\[
\theta_1\leq \frac{(d-1)+\max\al_i}{d} \theta_d.
\]
This contradicts \eqref{a-15}.
\end{proof}

\begin{lemma} \label{lem-A6}
The multi-valued map ${ {z}}\mapsto \Xi_z$ is upper semicontinuous
on $\partial \mathcal{S}$. That is,
for each ${ {z}}\in\partial \mathcal{S}$ there exists $\eps>0$ such that
\[
\Xi_x\subset \Xi_z\quad \text{ for all }{ {x}}\in B_\eps({ {z}})\cap\partial \mathcal{S}.
\]
\end{lemma}

\begin{proof} Take $z\in\partial\mathcal{S}$ and put 
\[
\Xi_z^{\mathrm c}:=\Xi_*\setminus\Xi_z. 
\]
Since 
\[
z\not\in \partial H(\xi) \ \ \text{ for any }\xi\in \Xi_z^{\mathrm c}
\]
and $\#\Xi_z^{\mathrm c}<\infty$, 
there is $\eps>0$ such that 
\[
\dist(z,\partial H(\xi))>\eps \ \ \text{ for all }\xi\in\Xi_z^{\mathrm c}. 
\]
That is, we have 
\[
B_\eps(z)\cap \partial H(\xi)=\emptyset \ \ \text{ for all }\ \xi\in \Xi_z^{\mathrm c},
\]
and, consequently,
\[
\Xi_x\subset \Xi_z \ \ \text{ for all }\ x\in B_\eps(z)\cap \partial\calS. 
\]
\end{proof}

\begin{lemma} \label{lem-A7}
There exists $\del>0$ such that for each $i\in\calI$,
\[
\xi_i=\al_i \quad \text{ and } \quad  \max_{j\in\calI} \xi_j=1 \quad
\text{ for all }\
{ {\xi}} \in \co\Xi_z \ \text{ and } \
{ {z}}\in\partial \mathcal{S}\cap F_{i,\del},
\]
where $F_{i,\del}=\{{ {x}}\in \Rd\mid x_j\geq -\del \ \text{ for all }\
j\in \calI, \
|x_i|\leq \del\}$.
\end{lemma}

\begin{proof}
Let ${ {z}}\in \partial \mathcal{S} \cap F_d$. By Lemma \ref{lem-A6}, there exists
$\eps:=\eps_z>0$ such that
\[
\Xi_x\subset \Xi_z \quad \text{ for all }\; { {x}}\in B_\eps({ {z}})\cap\partial \mathcal{S}.
\]
This yields
\[
\co \Xi_x \subset \co \Xi_z  \quad \text{ for all }\; { {x}}\in B_\eps({ {z}})\cap \partial \mathcal{S}.
\]
The family $\{B_{\eps}({ {z}})\}$ is an open covering of the compact set
$\partial \mathcal{S}\cap F_d$. Hence, there exists $\del_d>0$ such that
for any ${ {x}}\in \partial \mathcal{S}\cap F_{d,\del_d}$, there is
${ {z}}\in \partial \mathcal{S}\cap F_d$ such that
\[
\co \Xi_x \subset \co \Xi_z,
\]
which implies that
\[
\xi_d=\al_d \quad \text{ and } \quad \max_{i\in\calI}\xi_i=1 \quad
\text{ for all }\ { {\xi}}\in\co \Xi_x.
\]
Similarly, we may choose $\del_i>0$ for each $i\in\calI$ such that
for any ${ {x}}\in \partial \mathcal{S}\cap F_{i,\del_i}$,
\[
\xi_i=\al_i \quad \text{ and } \quad \max_{j\in\calI}\xi_j=1 \ \ \
\text{ for all }\ { {\xi}}\in\co \Xi_x.
\]
Thus, setting $\del:=\min_{i\in\calI}\del_i$, we conclude that
for any $i\in\calI$ and ${ {x}}\in \partial \mathcal{S}\cap F_{i,\del}$,
\[
\xi_i=\al_i \quad \text{ and } \quad \max_{j\in\calI}\xi_j=1 \quad
\text{ for all }\ { {\xi}}\in\co \Xi_x.
\]
\end{proof}

\subsection{Sign of $\mathcal{H}_i$}

Let us recall that for $i\in\calI$,
\begin{equation*}
\mathcal{H}_i( {p})=p_i-\alpha_i\max_{j\neq i}p_j^+.
\end{equation*}
Set
\[
E_{i,\del}=\{{ {x}}\in\Rd\mid \abs{x_i}\leq \del\} \quad \text{ for }\; i\in\calI.
\]

\begin{lemma} \label{lem-A8}
Let $\del>0$ be the constant from Lemma \ref{lem-A7}.
Let $i\in\calI$ and ${ {x}}\in \partial \mathcal{S}_\del \cap E_{i,\del}$.
Then
\[
\mathcal{H}_i({ {\gamma}})\leq 0 \quad \text{ for all } { {\gamma}}\in N({ {x}},\mathcal{S}_\del),
\]
where $N({ {x}},\mathcal{S}_\del)$ is given by \eqref{0000}.
\end{lemma}

\begin{proof}
We set $J:=\{j\in\calI\mid x_j=-\del\}$. Note that $J$ or $\Xi_x$ may be empty set.
We note that
\[
N({ {x}},\mathcal{S}_\del)=\sum_{j\in J}-\R_+ { {e}_j} + \R_+\, \co \Xi_x.
\]
Let ${ {\gamma}}\in N({ {x}},\mathcal{S}_\del)$. We choose $\lam_j\geq 0$, $\lam\geq 0$, and  $\xi=(\xi_1,\ldots,\xi_d)\in \co \Xi_x$
so that
\[
\gamma=\sum_{j\in J}-\lam_j e_j+\lam { {\xi}}.
\]
By Lemma \ref{lem-A7}, we have
\[
\xi_j=\al_j \quad \text{ for all }\; j\in J,
\]
and
\[
\xi_i=\al_i \quad\text{ and } \quad \max_{j\in \calI}\xi_j=1.
\]
In particular, we have $\xi_k=1$ for some $k\not\in J\cup \{i\}$.
Thus, we have
\[
{ {\gamma}}=\sum_{j\in J}(\lam \al_j-\lam_j){ {e}_j} + \lam  {e}_k
+\sum_{l\in\calI\setminus (J\cup\{k\})}
\lam \xi_l  {e}_l.
\]
Now, if $i\in J$, then
\[
\mathcal{H}_i({ {\gamma}}) =\lam\al_i-\lam_i-\al_i\lam=-\lam_i\leq 0,
\]
and, if  $i\not\in J$, then
$$
\mathcal{H}_i({ {\gamma}})=\lam \al_i-\al_i\lam=0.$$
\end{proof}

Set
\[
G_i=\{{ {x}}\in\Rd\mid x_i>-\del\}.
\]

\begin{lemma} \label{lem-A9}
Let $\del>0$ be as before. Let $i\in\calI$ and ${ {x}}\in \partial \mathcal{S}_\del \cap G_i$.
Then
\[
\mathcal{H}_i({ {\gamma}})\geq 0 \quad  \text{ for all }\quad { {\gamma}}\in N( {x},\mathcal{S}_\del).
\]
\end{lemma}

\begin{proof}
Set $J:=\{j\in\calI\mid x_j=-\del\}$. Note that $J$ may be an empty set
and that, since ${ {x}}\in G_i$, we have $i\not\in J$.
Let ${ {\gamma}}\in N({ {x}},\mathcal{S}_\del)$, and recall that
\[
N({ {x}},\mathcal{S}_\del)=\sum_{j\in J}\R_+(-{ {e}_j})+\R_+ \co \Xi_x.
\]
Choose $\lam_j\geq 0$, $\lam\geq 0$, and ${ {\xi}}\in \co \Xi_x$ so that
\[
{ {\gamma}}=\sum_{j\in J}-\lam_j{ {e}_j}+\lam { {\xi}},
\]
and observe that
\[
{ {\gamma}}=\sum_{j\in J}(\lam \xi_j-\lam_j){ {e}_j}
+\sum_{k\in \calI\setminus J}\lam {{\xi}}_k {e}_k
\]
and that
\[
\al_k\leq {{\xi}}_k\leq 1 \quad \text{ for all }k\in\calI.
\]
Hence,
\[
\mathcal{H}_i({ {\gamma}})=\gamma_i-\al_i\max_{k\neq i}\gamma_k^+
\geq \lam{{\xi}}_i-\al_i\max_{k\neq i}\lam \xi_k
\geq \lam \al_i-\al_i\lam\max_{k\neq i}\xi_k
\geq 0.
\]
\end{proof}

\begin{remark}
Combining Lemmas \ref{lem-A8} and \ref{lem-A9},
we see that for any ${ {x}}\in \partial \mathcal{S}_\del$, if $\abs{x_i}<\del$ for some $i$, then
\[
\mathcal{H}_i({ {\gamma}})=0 \quad \text{ for all }\; {\gamma}\in N({ {x}},\mathcal{S}_\del).
\]
\end{remark}

\subsection{Construction of $\varphi( {x})$}

For $0<\eps<\del$, 
let
\begin{equation*}
\mathcal{S}^\eps_\delta:=\left\{ {x}\in \Rd\ :\ \text{dist}\left( {x},
\mathcal{S}_\delta\right)\leq \eps\right\},
\end{equation*}
where $\delta$ is the constant from Lemma \ref{lem-A7}. It is clear that $\mathcal{S}^\eps_\delta$ satisfies the uniform interior sphere condition. Thus we have the following lemma.
\begin{lemma}\label{lem-A10}
The boundary $\partial\mathcal{S}^\eps_\delta$  is locally represented as the graph of a $\calC^{1,1}$ function.
\end{lemma}

\begin{proof}
Fix any $z\in\partial \calS_\del^\eps$.  We show that $\partial \calS_\del^\eps$ 
can be represented as the graph of a $\calC^{1,1}$ function in a neighborhood of $z$.
Making an orthogonal change of variables, we assume 
that $z=|z|e_d=(0,\ldots,0,|z|)$.  

Let $B'_R(x')$ denote the open ball in $\R^{d-1}$
of radius $R>0$ and center at $x'\in\R^{d-1}$ and let $B'_R:=B'_R(0)$.  
We choose $r>0$ so that $\bar{B}'_{3r} \times\{0\}
\subset \big(\calS_\del^\eps\big)^\circ$
Recall that $\calS_\del^\eps$ is a convex, compact, neighborhood of the origin of $\Rd$, note that, if $x' \in B'_{3r}$, then
the set $\{x_d\mid (x',x_d)\in \calS_\del^\eps\}$
is a closed, finite interval containing the origin of $\R$, and
set
\[
g(x')=\max\{x_d\mid (x',x_d)\in \calS_\del^\eps\}.
\]
Note that $0<g(x')\leq M$ for all $x'\in B'_{3r}$ and some constant $M>0$.

We show that $g\in \calC^{1,1}(B'_r)$. 
Since $\calS_\del^\eps$ is convex, $g$ is concave on $B'_{3r}$. 
This together with the boundedness
of $g$ implies that $g$ is Lipschitz continuous on $B'_{2r}$. Let $L>0$ be a 
Lipschitz bound of $g$ on $B'_{2r}$.  

Fix any $x'\in B'_r$.  
By the definition of $\calS_\del^\eps$ and the convexity of $\calS_\del$, 
there exists a unique $y\in \calS_\del$ such that
\[
\bar B_\eps(x',g(x'))\cap  \calS_\del=\{y\}.
\]
On the other hand, since $y\in\calS_\del$, we have 
\begin{equation}\label{A10-1}
B_\eps(y)\subset \calS_\del^\eps.
\end{equation}
Now, using a simple geometry and the Lipschitz property of $g$, we deduce
that 
\[
g(x')>y_d \ \ \text{ and } \ \ \frac{|x'-y'|}{g(x')-y_d}\leq L.
\] 
Indeed, the second estimate above can be obtained as follows: due to the continuity of 
$x'\mapsto (x', g(x'))\mapsto y$, it is enough to establish the above relation
at point $x'$ where $\grad g$ exists.
Convexity of $\calS_\del$ and $\calS^\eps_\del$ implies that $(x'-y', g(x')-y_d)\in N((x', g(x')), \calS^\eps_\del)$ and  $N((x', g(x')), \calS^\eps_\del)=\{\lambda(-\grad g(x'), 1)\colon \lambda\in\R_+\}$. Thus for some $\lambda>0$ we have $(x'-y', g(x')-y_d)=\lambda (-\grad g(x'), 1)$. Now it is easy to see that
the second estimate above holds.
Furthermore, we get  
\[
\eps^2=|x'-y'|^2+(g(x')-y_d)^2\geq (1+L^{-2})|x'-y'|^2. 
\]
Thus, setting $\theta:=L/\sqrt{1+L^2}\in (0,\,1)$, we have 
\begin{equation}\label{A10-2}
|x'-y'|\leq \theta\eps,
\end{equation}

Note that $V:=B'_r\cap B'_\eps(y')$ is a neighborhood of $x'$. Let $z'\in V$ 
and observe by \eqref{A10-1} that 
\[
|(z',g(z'))-y|\geq \eps=|(x',g(x'))-y|.  
\]
From this, noting that $g(z')>y_d$, we get
\begin{equation}\label{a-17}
g(z')\geq g(x')+\sqrt{\eps^2-|z'-y'|^2}-\sqrt{\eps^2-|x'-y'|^2}.
\end{equation}

Next  setting $h(z')=\sqrt{\eps^2-|z'-y'|^2}$ for 
$z'\in V$ and using \eqref{A10-2}, 
we calculate 
\[
D^2h(x')=-\frac{1}{\eps^2-|x'-y'|^2}\left(
\mathbb{I}'+\frac{(x'-y')\otimes (x'-y')}{\sqrt{\eps^2-|x'-y'|^2}}
\right),
\]
where $\mathbb{I}'$ denotes the $(d-1)\times(d-1)$ identity matrix,  
and moreover,
\[
D^2h(x')\geq- \frac{1}{\eps^2(1-\theta^2)}\left(
1+\frac{\eps\theta^2}{\sqrt{1-\theta^2}}
\right)\mathbb{I}'.
\]
This and \eqref{a-17} show that $g$ is semi-convex on $B'_r$, which together
with the concavity of $g$ ensures that $g\in \calC^{1,1}(B'_r)$.
\end{proof}

Let $\rho: \Rd\rightarrow[0,+\infty)$ be the gauge function of  $\mathcal{S}^\eps_\delta$, defined by
\begin{equation}\label{A10-3}
\rho(x):=\inf\{\lambda> 0\mid{x}\in\lambda\mathcal{S}^\eps_\delta\},
\end{equation}
where
\[\lambda\mathcal{S}^\eps_\delta:=\{\lambda {\xi}\mid {\xi}\in
\mathcal{S}^\eps_\delta\}.\]
It is well known that $\rho(\cdot)$ is a positively homogeneous, of degree one, convex function. In particular, we have $\rho( {0})=0$. It is clear to see that 
$\rho(y)>0$ at $ {y}\neq  {0}$. 

\begin{lemma} \label{lem-A11} The function $\rho$ is locally $\calC^{1,1}$ on $\R^d\setminus\{0\}$ 
and for $y\in\R^d\setminus\{0\}$, 
$D\rho( {y})\in N( {y}/\rho( {y}),\mathcal{S}^\eps_\delta)$,
where $N( {y}/\rho( {y}),\mathcal{S}^\eps_\delta)$ denotes the normal cone at $ {y}/\rho( {y})\in \partial \mathcal{S}^\eps_\delta$.
\end{lemma}

\begin{proof}  It follows from the definition of $\rho$ that for any $y\in\R^d\setminus\{0\}$ and $\lam>0$,  we have 
\begin{equation}\label{A11-1}
\lam=\rho(y)\ \ \text{ if and only if } \  \  
y/\lam\in\partial S_\del^\eps. 
\end{equation} 

Fix any $z\in\R^d\setminus\{0\}$ and show that $\rho$ is a $\calC^{1,1}$ 
function near $z$. Up to an orthogonal change of variables, one can assume that $ {z}=\abs{z} {e}_d$. 
It follows from Lemma \ref{lem-A10} that there exists a neighborhood  $V$ of the origin in $\R^{d-1}$ and a function $g\in \calC^{1,1}(V)$ such that for any $x'\in V$ and $x_d>0$,
\begin{equation}\label{A11-2}
(x',x_d) \in\partial S_\del^\eps \ \ \text{ if and only if } 
\ \ x_d=g(x').
\end{equation}

Set $\lam_z=\rho(z)$. By \eqref{A11-1} and \eqref{A11-2}, we have 
\[
z_d/\lam_z =g(z'/\lam_z)=g(0). 
\]
Set $I=(\lam_z/2,\,2\lam_z)$ and $W=(\frac{2}{\lam_z})V\times (0,\,\infty)$,
and note that $W\times I$ is a neighborhood of $(z,\lam_z)\in\R^{d+1}$. 
For any $y=(y',y_d)\in W$, in view of \eqref{A11-1}, 
we intend to find $\lam\in I$ such that 
\[
y/\lam \in\partial S_\del^\eps. 
\] 
Since $y'/\lam\in V$, the above inclusion reads 
\[
y_d/\lam=g(y'/\lam). 
\]
We define the function $F\,:\,W\times I\to \R$ by
\[
f(y,\lam)=y_d/\lam-g(y'/\lam),
\]
note that $f(z,\lam_z)=0$, and calculate that 
\[
\frac{\partial f}{\partial \lam}(z,\lam_z)=-\frac{z_d}{\lam_z^2}<0. 
\]
By the implicit function theorem, there are a neighborhood $U\subset W$ of $z$ 
and a $\calC^{1,1}$ function $\Lambda \,:\, U\to I$ such that 
$f(y,\Lambda(y))=0$ for all $y\in U$. In view of \eqref{A11-1} and \eqref{A11-2}, 
we see that $\rho(y)=\Lambda(y)$ for all $y\in U$, which completes the proof. 

Now, let $y\in\R^d\setminus\{0\}$ and $p=D\rho(y)$. Note that  
$y/\rho(y)\in\partial \calS_\del^\eps$ and $p=D\rho(y/\rho(y))$ 
and that for any $x\in\calS_\del^\eps$,
\[
1\geq \rho(x)\geq\rho(y/\rho(y))+p\cdot (x-y/\rho(y))=1+p\cdot (x-y/\rho(y)), 
\]  
that is, we have 
\[
p\cdot (x-y/\rho(y))\leq 0 \ \ \text{ for all } x\in\calS_\del^\eps.
\]
This shows that
\[
D\rho(y)=p\in N(y/\rho(y),\calS_\del^\eps). \qedhere
\]
\end{proof}

\begin{lemma} \label{lem-12} Let $x\in\partial\calS_\del^\eps$ and $i\in\calI$. 
For any $\gamma\in N(x,\calS_\del^\eps)$,  
\[
\calH_i(\gamma)
\begin{cases}
\leq 0 & \text{ if }\ x_i <\del-\eps,\\[3pt]
\geq 0 & \text{ if }\ x_i >-(\del-\eps).
\end{cases}
\]
\end{lemma}

\begin{proof} Fix any $x\in\partial\calS_\del^\eps$. 

Since $\rho(x)=1$, 
$\rho(z)\leq 1$ for all $z\in\calS_\del^\eps$, and 
$\rho$ is convex, we have 
\[
1\geq \rho(z)\geq \rho(x)+D\rho(x)\cdot(z-x) \ \ \text{ for all } 
z\in\calS_\del^\eps,
\] 
which implies that $D\rho(x)\in N(x,\calS_\del^\eps)$. 
Since $\rho$ is differentiable at $x$ and 
$D\rho(x)\not=0$, it is easily seen that  
\begin{equation}\label{lem-12-1}
N(x,\calS_\del^\eps)=\R_+D\rho(x). 
\end{equation}

Set $y:=x-\eps D\rho(x)/|D\rho(x)|$ and observe that
\begin{equation}\label{lem-12-2}
\bar B_\eps(x)\cap \partial\calS_\del=\{y\},
\end{equation}
which implies that $\,
\bar B_\eps(x)\cap \calS_\del=\{y\}$, 
and therefore 
\begin{equation}\label{lem-12-3}
D\rho(x)\in N(y,\calS_\del).
\end{equation}
Indeed, it is clear that 
\[
\bar B_\eps(x)\cap \partial\calS_\del\not=\emptyset. 
\]
Note moreover that if $\,z\in\bar B_\eps(x)\cap \partial\calS_\del$,
then $\,
B_\eps(z)\subset \calS_\del^\eps\,$ 
and, according to \eqref{lem-12-1},
\[
D\rho(x)\cdot (w-x)\leq 0 \ \ \text{ for all }w\in B_\eps(z).
\] 
This last inequality readily yields  
\[
z-x=-\eps D\rho(x)/|D\rho(x)|,
\]
that is, $z=y$, and we conclude that \eqref{lem-12-2} holds.   

Fix any $\gamma\in N(x,\calS_\del^\eps)$ and $i\in\calI$. 
By \eqref{lem-12-1} and \eqref{lem-12-3}, we have 
\[
\gamma=(|\gamma|/|D\rho(x)|)D\rho(x)\in N(y,\calS_\del).
\]
Thus, noting that 
\[
y_i\begin{cases}
\leq x_i+|x-y|<\del &\text{ if }\ x_i<\del-\eps,\\[3pt]
\geq x_i-|x-y|>-\del & \text{ if }\ x_i>-(\del-\eps),
\end{cases}
\]
we may apply Lemmas \ref{lem-A8} and \ref{lem-A9}, to get 
\[
\calH_i(\gamma)
\begin{cases}
\leq 0&\text{ if } \ x_i<\del-\eps,\\[3pt]
\geq 0&\text{ if } \ x_i>-(\del-\eps),
\end{cases}
\]
which completes the proof. 
\end{proof}

\begin{proof}[Proof of Theorem \ref{test-f}] 
As noted before, we know that when $d=1$,  the function 
$\ph(x):=x^2$ has all the required properties. 

In what follows we assume that $d>1$, and we define
\begin{equation*}
\varphi( {x}):= \rho^2( {x}) \quad  \text{ for } \ x\in \Rd,
\end{equation*}
where $\rho$ is the function given by \eqref{A10-3}. It is clear that 
$\ph(x)>0$ for all $x\in\R^d\setminus\{0\}$ and $\ph$ is positively homogeneous 
of degree two. 

Next we show that $\ph$ is globally $\calC^{1,1}$ on $\R^d$. 
Since $\rho$ is locally $\calC^{1,1}$ on $\R^d\setminus\{0\}$, the function 
$\ph$ is also locally $\calC^{1,1}$ on $\R^d\setminus\{0\}$. 
Let $x,y\in\R^d\setminus\{0\}$ and in view of the homogeneity property of $\ph$ we obtain that 
\[\begin{aligned}
D\ph(x)-D\ph(y)&\,= |x|D\ph(x/|x|)-|y|D\ph(y/|y|)
\\&\,=|x|\big(D\ph(x/|x|)-D\ph(y/|y|)\big)+(|x|-|y|)D\ph(y/|y|).
\end{aligned}
\]
We choose a constant $C>0$ so that 
\[
|D\ph(p)|\leq C 
\ \ \text{ and } \ \ 
|D\ph(p)-D\ph(q)|\leq C|p-q| \ \ \text{ for all }p,q\in \partial B_1. 
\]
Combining these yields
\[\begin{aligned}
|D\ph(x)-D\ph(y)|&\leq 
|x|\,|D\ph(x/|x|)-D\ph(y/|y|)|+\big|(|x|-|y|)D\ph(y/|y|)\big|
\\&\leq C|x|\,\big|x/|x|-y/|y|\big|+C|x-y|
\\&\leq 2C|x-y|+C|x-y|=3C|x-y|.
\end{aligned}
\]
This shows that for any $x,y\in\R^d$,
\[
|D\ph(x)-D\ph(y)|\leq 3C|x-y|.
\]
Accordingly, the function $\ph$ is globally $\calC^{1,1}$ on $\R^d$. 

Let $x\in\R^d\setminus\{0\}$ and $i\in\calI$, and assume that $x_i\geq 0$. 
By Lemma \ref{lem-A11}, we have 
\[
D\ph(x)=2\rho(x)D\rho(x)\in N(x/\rho(x),\calS_\del^\eps). 
\]
Noting that $x/\rho(x)\in\partial \calS_\del^\eps$, we see from Lemma \ref{lem-12} 
that 
\[
\calH_i(D\ph(x))\geq 0.  
\]
Thus, noting that $D\ph(0)=0$ and, hence, $\calH_i(D\ph(0)=0$ for all $i\in\calI$, 
we may conclude that for any $x\in\R^d$ and $i\in\calI$,
\[
\calH_i(D\ph(x))\geq 0 \ \ \text{ if }\ x_i\geq 0.  
\] 
An argument similar to the above shows that for any $x\in\R^d$ and $i\in\calI$,
\[
\calH_i(D\ph(x))\leq 0 \ \ \text{ if }\ x_i\leq 0.  
\] 
This completes the proof. 
\end{proof}

\begin{remark}\label{R-smooth}
The $\calC^{1,1}$ regularity of 
test functions is sufficient to achieve the comparison theorem as is shown in the proof of Theorem \ref{T-comp}. 
Alternatively, one can build smooth test functions 
by using the mollification techniques.   
\end{remark}

\appendix
\section{}\label{S-appen}
Let us first recall the Skorokhod problem which will be used in the analysis below.
\begin{definition}[Skorokhod Problem]
Given $\psi\in \calC([0, \infty): \R)$ with $\psi(0)\geq 0$, a pair $(\phi,\eta)\in \bigl(\calC([0, \infty):\R_+)\bigr)^2$
is said to solve the Skorokhod Problem (SP) for $\psi$ on the domain $[0, \infty)$, if
\begin{enumerate}
\item $\phi(t)=\psi(t)+\eta(t), \quad t\geq 0$,
\item $\phi(0)=\psi(0),\quad t\geq0$,
\item $\eta$ is non-negative, non-decreasing and
$\int_0^t \phi (s)d\eta(s)=0, \quad \forall \; t\geq 0$.
\end{enumerate}
\end{definition}
In fact, there is always a unique pair $(\phi, \eta)$ solving the Skorokhod Problem for a given $\psi$,
given by
\[
\eta(t)=\sup_{s\in[0, t]}(-\psi(s))^+,
\qquad
\phi(t)=\psi(t)+\eta(t),\qquad t\ge0.
\]
Recall that for $\alpha=(\alpha_1, \ldots, \alpha_d)\in\Rd_+$ with $\max_i\alpha_i<1$, we have
$$\M(\alpha)=\{M=[\mathbb{I}-P]\in \mat \, :\, P_{ii}=0, P_{ij}\geq 0\,\, \text{and}  \sum_{j:j\neq i}P_{ji}\leq \alpha_i\}.$$
In this section we show that for any $v\in\Uadm$ and given $x\in \Rd_+$, the following reflected diffusion has a unique adapted strong solution $(X, Y)\in \calC([0, \infty): \Rd_+)\times \calCup_0([0, \infty): \Rd_+)$,
\begin{align}\label{a-1}
X(t) &= x + \int_0^t b(X(s))ds + \Sigma W(t) + \int_{[0, t]} v(s) dY(s), \quad t\geq 0,\; \text{a.s.,}
\\
\int_{[0, \infty)} X_i(s)dY_i(s) & =0  \quad \text{for all}\,\; i\in\calI,\nonumber
\end{align}
where $W$ is a standard $d$-dimensional Brownian motion on a given filtered probability space $(\Omega, \calF, \{\calF_t\}, \Prob)$.
Here $\Uadm$ denotes the set of all admissible controls $v$ where $v$ takes values in $\M(\al)$. Before we state the result let us define the following quantity,
\begin{align}
\beta_i(t) = \sup_{0\leq s\leq t}\max\{0, (-\Sigma W(s))_i\}.\label{a-2}
\end{align}
The following result is a modified version of the results obtained in \cite{ram}.
\begin{theorem}\label{Thm-existence}
For any $v\in\Uadm$, \eqref{a-1} has a unique adapted strong solution $(X, Y)$ taking values in $\calC([0, \infty): \Rd_+)\times \calCup_0([0, \infty): \Rd_+)$ and
\begin{equation}\label{a-3}
\sum_{i\in\calI} Y_i(t)\leq \frac{1}{1-\max_i\al_i}\Big(\abs{b}_1 t + \sum_{i\in\calI}\beta_i(t)\Big)  \quad \text{for all}\;\; t\geq 0, \; \text{a.s.,}
\end{equation}
where $\beta$ is given by \eqref{a-2} and $\abs{b}_1=\sum_{i\in\calI}\sup_{x\in\Rd_+}\abs{b_i(x)}$.
\end{theorem}

We refer, for instance, to \cite{lions-sznitman} for general 
results for reflected diffusion processes on smooth domains. 

\begin{proof}
The key ideas of this proof are borrowed from \cite{ram}. Fix $T>0$. To show the existence of a unique adapted strong solution it enough to show the existence of a unique adapted strong solution
on the the interval $[0, T]$. Thus we consider the space $\calC([0, T]: \Rd)\times \calCup_0([0, T]: \Rd_+)$. For $(x, y)\in \calC([0, T]: \Rd_+)\times \calCup_0([0, T]: \Rd_+)$, we define
\begin{align*}
\bnorm{x}_T &= \sum_{i\in\calI} \sup_{s\in [0, T]}\abs{x_i(s)},
\\
\bvnorm{y}_T &= \sum_{i\in\calI} \bvnorm{y_i}_T,
\end{align*}
where we recall that $\bvnorm{\cdot}_s$ denotes the bounded variation norm on the interval $[0, s]$. For positive $\gamma_1, \gamma_2$, we endow the space $\calC([0, T]: \Rd_+)\times \calCup_0([0, T]: \Rd_+)$ with the metric
$$\D((x, y), (\bar{x}, \bar{y}))=\gam_1\bnorm{x-\bar{x}}_T +\gam_2 \bvnorm{y-\bar{y}}_T,$$
where  $(x, y),\, (\bar{x}, \bar{y})\in \calC([0, T]: \Rd_+)\times \calCup_0([0, T]: \Rd_+)$. It is easy to check that for any positive $\gam_i, i=1,2,$ $\calC([0, T]: \Rd_+)\times \calCup_0([0, T]: \Rd_+)$ forms a complete metric space with the metric $\D$.
For $(x, y)\in \calC([0, T]: \Rd_+)\times \calCup_0([0, T]: \Rd_+)$ we define two maps $\calT: \calC([0, T]: \Rd_+)\times \calCup_0([0, T]: \Rd_+)\to \calCup_0([0, T] : \Rd)$ and $\calS : \calC([0, T]: \Rd_+)\times \calCup_0([0, T]: \Rd_+)\to \calC([0, T]: \Rd_+)$
as follows:
\begin{align*}
\calT_i(x, y)(t) & = \sup_{0\leq s\leq t} \max\{0, -U_i(x, y)(s)\},
\\
\calS_i(x, y)(t) &= U_i(x, y)(t) + \calT_i(x, y)(t), 
\\
\text{where} \;\; U_i(x, y)(t) &=x_i +\int_0^t b_i(x(s))ds + (\Sigma W)_i(t) +\sum_{j\neq i}\int_0^tv_{ij}(s)dy_j(s).
\end{align*}
From the property of 1-dimensional Skorokhod problem we note that any fixed point of $(\calS, \calT)$ solves \eqref{a-1}.
The idea of the proof is choose suitable $\gam_i, T_0$ so that the map $(\calS, \calT): \calC([0, T_0]: \Rd_+)\times \calCup_0([0, T_0]: \Rd_+)\to \calC([0, T_0]: \Rd)\times \calCup_0([0, T_0]: \Rd_+)$ becomes a contraction almost surely. Once we
have uniqueness on the interval $[0, T_0]$ one can extend the result on $[0, T]$ be standard time shifting argument. Now we fix a sample point in $(\Omega, \calF, \Prob)$. Then we view $X_0, W$ as deterministic elements. Let
$K$ be the Lipschitz constant of $b$. Let $(x, y), (\bar{x}, \bar{y})\in \calC([0, T]: \Rd)\times \calCup_0([0, T]: \Rd_+)$. Using the lemma of Shashiasvili \cite[pp. 170-175]{shashiasvili} concerning variational distance between maximal functions, we get
\begin{align*}
\bvnorm{\calT_i(x, y)-\calT_i(\bar{x}, \bar{y})}_T
& \leq \bvnorm{U_i(x, y)-U_i(\bar{x}, \bar{y})}_T
\\
&\leq \int_0^T\abs{b_i(x(s))-b_i(\bar{x}(s))}ds + \sum_{j\neq i}\int_0^T(-v_{ij}(s))\abs{dy_j-d\bar{y}_j}(s),
\end{align*}
where we have used the property that $v_{ij}\leq 0$. We sum over $i$ to obtain
\begin{align}\label{a-4}
\bvnorm{\calT(x, y)-\calT(\bar{x}, \bar{y})}_T & \leq KdT\, \bnorm{x-\bar{x}}_T +\sum_{j=1}^d\int_0^T\sum_{i\neq j} (-v_{ij}(s))\abs{dy_j-d\bar{y}_j}(s)
\nonumber
\\
& \leq KdT\, \bnorm{x-\bar{x}}_T +\max_{i}{\al_i}\; \sum_{j=1}^d\int_0^T \abs{dy_j-d\bar{y}_j}(s)\nonumber
\\
&= KdT\, \bnorm{x-\bar{x}}_T +\max_{i}{\al_i}\; \bvnorm{y-\bar{y}}_T.
\end{align}
Similarly,
\begin{align}\label{a-5}
\bnorm{\calS(x, y)-\calS(\bar{x}, \bar{y})}_\infty \leq 2\, KdT \bnorm{x-\bar{x}}_T + 2\, \max_i\al_i \; \bvnorm{y-\bar{y}}_T.
\end{align}
Therefore combining \eqref{a-4} and \eqref{a-5} we obtain
\begin{align*}
& \D((\calS(x, y), \calT(x, y)),((\calS(\bar{x}, \bar{y}), \calT(\bar{x}, \bar{y})))
\\
& \leq KdT \frac{(2\gam_1+\gam_2)}{\gam_1}\, \gam_1\bnorm{x-\bar{x}}_T
 +
\max_i\al_i\frac{(2\gam_1+\gam_2)}{\gam_2}\, \gam_2 \bvnorm{y-\bar{y}}_T.
\\
&\leq \max\{KdT \frac{(2\gam_1+\gam_2)}{\gam_1}, \max_i\al_i\frac{(2\gam_1+\gam_2)}{\gam_2}\} \D((x, y), (\bar{x}, \bar{y})).
\end{align*}
Now we first choose $\gam_1, \gam_2$  to satisfy $\max_i\al_i\frac{(2\gam_1+\gam_2)}{\gam_2}<1$ (use the fact that $\max_i\al_i<1$) and the choose $T$ small enough to satisfy $KdT \frac{(2\gam_1+\gam_2)}{\gam_1}<1$. Hence we have
$$\D((\calS(x, y), \calT(x, y)),((\calS(\bar{x}, \bar{y}), \calT(\bar{x}, \bar{y}))) <\varrho\, \D((x, y), (\bar{x}, \bar{y}))\quad \text{for some}\,\; \varrho<1.$$
Therefore one can use contraction mapping theorem to get the unique fixed point in $\calC([0, T_0]: \Rd)\times \calCup_0([0, T_0]: \Rd)$ that  solves \eqref{a-1}.

Now we prove \eqref{a-3}. Let $(X, Y)\in \calC([0, \infty): \Rd_+)\times \calCup_0([0, \infty): \Rd_+)$ be the solution of \eqref{a-1}. Then
$$Y_i(t)=\sup_{0\leq s\leq t} \max\{0, -U_i(X, Y)(s)\} \quad \text{for all}\quad i\in\calI,$$
where
$$U_i(X, Y)(t) = x_i + \int_0^t b_i(X_s)ds + (\Sigma W)_i + \sum_{j\neq i}\int_0^t v_{ij}(s)dY_j(s).$$
Since $x\in\Rd_+$ we have from \eqref{a-2}
\begin{align*}
-U_i(X, Y)(t)\leq t\sup_{x\in\Rd_+}\abs{b_i(x)} + \beta_i(t) -\sum_{j\neq i}\int_0^t v_{ij}(s)dY_j(s).
\end{align*}
Thus using the fact $v_{ij}\leq 0$ for $i\neq j$ we obtain
\begin{align*}
Y_i(t)\leq t\sup_{x\in\Rd_+}\abs{b_i(x)} + \beta_i(t) -\sum_{j\neq i}\int_0^t v_{ij}(s)dY_j(s)
\end{align*}
Summing over $i$ and using the property that $v(s)\in\M(\al)$ we have
$$\sum_{i\in\calI} Y_i(t)\leq t\, \abs{b}_1 + \sum_{i} \beta_i(t) + \max_{i}\al_i\, \sum_{i\in\calI} Y_i(t),$$
 and this proves \eqref{a-3}.
\end{proof}

Recall $J$ and $V$ from \eqref{2.5} and \eqref{3}. Also recall $\sigma$ from Proposition~\ref{prop-dpp}.
\begin{lemma}\label{lem-A1}
Let $\eps>0$ and $x\in\Rd$. Then for every $v\in\Uadm$, there exists a control $\tilde{v}\in\Uadm$ such that
$$\Exp_x\Big[\int_{\sigma\wedge t}^\infty e^{-\beta\, s}(\ell(X(s))ds +\sum_{i\in\calI} h_i(\tilde{v})dY_i)\, \Big| \mathcal{F}_{\sigma\wedge t}\Big]
 \leq e^{-\beta\, \sigma\wedge t}V(X(\sigma\wedge t)) + \eps
\quad \text{a.s.}$$
\end{lemma}

\begin{proof}
Recall the $\sigma_r$ is the exit time from the open ball $\B_r(x)$. Now using the continuity property of $V$ (Lemma~\ref{lem-cont}) and \eqref{13}
we can find $\del>0$ so that
\begin{equation}\label{a-10}
\sup_{v\in\Uadm}\babs{J(y, v)-J(\bar{y}, v)} + \abs{V(y)-V(\bar{y}} <\eps \quad \text{ for all }\; y, \bar{y}\in \bar{\B}_r(x),
\end{equation}
and $\abs{y-\bar{y}}\leq\del$. Now consider a finite Borel partition $\{D_k\}_{k\geq 1}$ of $\bar{\B}_r(x)$ such that $\text{diam}(D_j)<\del$.
Choose $y_k\in D_k$. Then by definition we have $v^k\in\Uadm$ satisfying 
$$J(y_k, v^k) \leq V(y_k) +\eps.$$
Combining with \eqref{a-10} we see that for any $y\in D_k$, we have
\begin{equation}\label{a-11}
J(y, v^k) \leq J(y_k, v^k) + \eps \leq V(y_k) + 2\eps \leq V(y) + 3\eps.
\end{equation}
Since $v^k$ is adapted to $\{\calF_t\}$ we have a progressively measurable map $\Psi_j : [0, \infty)\times\calC([0, \infty): \Rd)\to \M$
such that
$$v^k(t) = \Psi_k(t, W(\cdot\wedge t)), \quad {a.s.}\quad \forall t\geq 0.$$
Now for any $v\in\Uadm$ and the corresponding solution $(X, Y)$ define
\[
\tilde{v}(s)=\left\{
\begin{array}{lll}
v(s) & \text{if}\quad s\leq \sigma\wedge t,
\\
\Psi_k(s-\sigma\wedge t, W(\cdot\wedge s)-W(\sigma\wedge t)) & \text{if} \quad s\, >\sigma\wedge t, \;\text{and}\; X(\sigma\wedge t)\in D_k.
\end{array}
\right.
\]
By strong uniqueness we see that $\sigma\wedge t$ does not change with the new control $\tilde{v}$. Then
denoting $\bar{Y}=Y(\sigma\wedge t + \cdot)-Y(\sigma\wedge t)$, we obtain
\begin{align*}
&\Exp_x\Big[\int_{\sigma\wedge t}^\infty e^{-\beta\, s}(\ell(X(s))ds +\sum_{i\in\calI} h_i(\tilde{v})dY_i)\, \Big| \mathcal{F}_{\sigma\wedge t}\Big]
\\
&= e^{-\beta\, \sigma\wedge t} \Exp\Big[\int_{0}^\infty e^{-\beta\, s}\big(\ell(X(\sigma\wedge t + s))ds
+\sum_{i\in\calI} h_i(\tilde{v}(\sigma\wedge t + s))d \bar{Y}_i\big) \,
\Big| \mathcal{F}_{\sigma\wedge t}\Big]
\\
&= e^{-\beta\, \sigma\wedge t}\, J(X(\sigma\wedge t), \tilde{v}(\sigma\wedge t + \cdot))
\\
&\leq e^{-\beta\, \sigma\wedge t}\, V(X(\sigma\wedge t)) + 3\eps,
\end{align*}
where in the last inequality we use \eqref{a-11}.
\end{proof}

\section*{Acknowledgement} 
Anup Biswas acknowledges the hospitality of the Department of Electrical Engineering (specially, Prof. Rami Atar) in Technion while he was visiting at the early stages of this work. 
This research of Anup Biswas was supported in part by a INSPIRE faculty fellowship.
The research of Hitoshi Ishii was supported in part by the JSPS grants, (KAKENHI \#26220702 and \#16H03948). The research of Subhamay Saha was supported in part by ISF (grant 1315/12). The research of Lin Wang was supported in part by NSF of China (grant 11401107) and the JSPS grant, (KAKENHI \#26220702) and he also would like to thank the second author 
for the warm and thoughtful invitation to Waseda University.

\end{document}